\documentclass[12pt]{amsart}
\usepackage{mathrsfs}
\usepackage{cases}
\usepackage{epic}
\usepackage{amsfonts}
\usepackage{graphicx}
\usepackage{amsmath}
\usepackage{amssymb, upgreek}
\usepackage{bm}
\usepackage{latexsym,todonotes}
\usepackage{pdflscape}
\usepackage[all,cmtip]{xy}
\usepackage{color}
\usepackage{colordvi}
\usepackage{multicol}
\usepackage[normalem]{ulem}
\usepackage{comment}
\usepackage{tikz}
\usetikzlibrary{matrix, decorations.pathreplacing, positioning}

\usepackage[linktocpage=true]{hyperref}
\hypersetup{colorlinks,linkcolor=blue,urlcolor=cyan,citecolor=blue}

\begin{document}
	\input xy
	\xyoption{all}
	
	\numberwithin{equation}{section}
	\allowdisplaybreaks
	\newcommand{\proj}{\operatorname{proj.}\nolimits}
	\newcommand{\rad}{\operatorname{rad}\nolimits}
	\newcommand{\Irr}{\operatorname{Irr}\nolimits}
	\newcommand{\soc}{\operatorname{soc}\nolimits}
	\newcommand{\ind}{\operatorname{inj.dim}\nolimits}
	\newcommand{\id}{\operatorname{id}\nolimits}
	\newcommand{\Mod}{\operatorname{Mod}\nolimits}
	\newcommand{\R}{\operatorname{R}\nolimits}
	\newcommand{\End}{\operatorname{End}\nolimits}
	\newcommand{\ob}{\operatorname{Ob}\nolimits}
	\newcommand{\Ht}{\operatorname{Ht}\nolimits}
	\newcommand{\cone}{\operatorname{cone}\nolimits}
	\newcommand{\rep}{\operatorname{rep}\nolimits}
	\newcommand{\Ext}{\operatorname{Ext}\nolimits}
	\newcommand{\Tor}{\operatorname{Tor}\nolimits}
	\newcommand{\Hom}{\operatorname{Hom}\nolimits}
	\newcommand{\Pic}{\operatorname{Pic}\nolimits}
	\newcommand{\aut}{\operatorname{Aut}\nolimits}
	\newcommand{\Fac}{\operatorname{Fac}\nolimits}
	\newcommand{\Div}{\operatorname{Div}\nolimits}
	\newcommand{\rank}{\operatorname{rank}\nolimits}
	\newcommand{\Len}{\operatorname{Length}\nolimits}
	\newcommand{\RHom}{\operatorname{RHom}\nolimits}
	\newcommand{\Iso}{\operatorname{Iso}\nolimits}
	\newcommand{\Coh}{\operatorname{coh}\nolimits}
	\newcommand{\Qcoh}{\operatorname{Qch}\nolimits}
	\newcommand{\inj}{\operatorname{inj.dim}\nolimits}
	\newcommand{\dimv}{\operatorname{\underline{\dim}}\nolimits}
	\newcommand{\res}{\operatorname{res}\nolimits}
	\def \TT{\mathbf{T}}
	\def \U{\mathbf U}
	\def \tUB{{}^{\texttt{B}}\tU}
	\def \UB{{}^{\texttt{B}}\U}
	\newcommand{\indim}{\operatorname{inj.dim}\nolimits}
	\newcommand{\pdim}{\operatorname{proj.dim}\nolimits}
	\def \dbl{\rm dbl}
	\def \sqq{\mathbf v}
	\newcommand{\Aut}{\operatorname{Aut}\nolimits}
	\newcommand{\supp}{\operatorname{supp}\nolimits}
	
	\def \bm{\mathbf m}
	\def \Id{\mathrm{Id}}
	\def \QJ {Q_{\texttt{J}}}
	\newcommand{\fgm}{{\rm mod}^{{\rm fg}}}
	\newcommand{\fgmz}{\mathrm{mod}^{{\rm fg},\Z}}
	\newcommand{\fdmz}{\mathrm{mod}^{{\rm nil},\Z}}

	\newcommand{\Cp}{\operatorname{Cp}\nolimits}
	\newcommand{\coker}{\operatorname{Coker}\nolimits}
	\newcommand{\Ab}{{\operatorname{Ab}\nolimits}}
	\newcommand{\Cone}{{\operatorname{Cone}\nolimits}}
	\newcommand{\pd}{\operatorname{proj.dim}\nolimits}
	\newcommand{\sdim}{\operatorname{sdim}\nolimits}
	\newcommand{\add}{\operatorname{add}\nolimits}
	\newcommand{\pr}{\operatorname{pr}\nolimits}
	\newcommand{\oR}{\operatorname{R}\nolimits}
	\newcommand{\oL}{\operatorname{L}\nolimits}
	
	\newcommand{\tU}{\operatorname{\widetilde{\bf U}}\nolimits}
	
	\newcommand{\Perf}{{\mathfrak Perf}}
	\newcommand{\cc}{{\mathcal C}}
	\newcommand{\ce}{{\mathcal E}}
	\newcommand{\cs}{{\mathcal S}}
	\newcommand{\cf}{{\mathcal F}}
	\newcommand{\cx}{{\mathcal X}}
	\newcommand{\cy}{{\mathcal Y}}
	\newcommand{\cl}{{\mathcal L}}
	\newcommand{\ct}{{\mathcal T}}
	\newcommand{\cu}{{\mathcal U}}
	\newcommand{\cm}{{\mathcal M}}
	\newcommand{\cv}{{\mathcal V}}
	\newcommand{\ch}{{\mathcal H}}
	\newcommand{\ca}{{\mathcal A}}
	\newcommand{\mcr}{{\mathcal R}}
	\newcommand{\cb}{{\mathcal B}}
	\newcommand{\ci}{{\mathcal I}}
	\newcommand{\cj}{{\mathcal J}}
	\newcommand{\cp}{{\mathcal P}}
	\newcommand{\cg}{{\mathcal G}}
	\newcommand{\cw}{{\mathcal W}}
	\newcommand{\co}{{\mathcal O}}
	\newcommand{\cd}{{\mathcal D}}
	\newcommand{\ck}{{\mathcal K}}
	\newcommand{\calr}{{\mathcal R}}
	
	\def \fg{{\mathfrak g}}
	\newcommand{\ol}{\overline}
	\newcommand{\ul}{\underline}
	\newcommand{\cz}{{\mathcal Z}}
	\newcommand{\st}{[1]}
	\newcommand{\ow}{\widetilde}
	\newcommand{\pic}{\operatorname{Pic}\nolimits}
	\newcommand{\Spec}{\operatorname{Spec}\nolimits}
	\newtheorem{theorem}{Theorem}[section]
	\newtheorem{acknowledgement}[theorem]{Acknowledgement}
	\newtheorem{algorithm}[theorem]{Algorithm}
	\newtheorem{axiom}[theorem]{Axiom}
	\newtheorem{case}[theorem]{Case}
	\newtheorem{claim}[theorem]{Claim}
	\newtheorem{conclusion}[theorem]{Conclusion}
	\newtheorem{condition}[theorem]{Condition}
	\newtheorem{conjecture}[theorem]{Conjecture}
	\newtheorem{construction}[theorem]{Construction}
	\newtheorem{corollary}[theorem]{Corollary}
	\newtheorem{criterion}[theorem]{Criterion}
	\newtheorem{propdef}[theorem]{Definition-Proposition}
	
	\newtheorem{definition}[theorem]{Definition}
	\newtheorem{example}[theorem]{Example}
	\newtheorem{exercise}[theorem]{Exercise}
	\newtheorem{lemma}[theorem]{Lemma}
	\newtheorem{notation}[theorem]{Notation}
	\newtheorem{problem}[theorem]{Problem}
	\newtheorem{proposition}[theorem]{Proposition}
	\newtheorem{solution}[theorem]{Solution}
	\newtheorem{summary}[theorem]{Summary}
	\newtheorem*{thm}{Theorem}
	\newcommand{\qbinom}[2]{\begin{bmatrix} #1\\sharp 2 \end{bmatrix} }
	
	\newtheorem{innercustomthm}{{\bf Theorem}}
	\newenvironment{customthm}[1]
	{\renewcommand\theinnercustomthm{#1}\innercustomthm}
	{\endinnercustomthm}
	
	\newtheorem{innercustomcor}{{\bf Corollary}}
	\newenvironment{customcor}[1]
	{\renewcommand\theinnercustomcor{#1}\innercustomcor}
	{\endinnercustomthm}
	
	\newtheorem{innercustomprop}{{\bf Proposition}}
	\newenvironment{customprop}[1]
	{\renewcommand\theinnercustomprop{#1}\innercustomprop}
	{\endinnercustomthm}
	
	\theoremstyle{remark}
	\newtheorem{remark}[theorem]{Remark}
	
	\def \bfk{\mathbf k}
	\def \bp{{\mathbf p}}
	\def \bA{{\mathbf A}}
	\def \bL{{\mathbf L}}
	\def \bF{{\mathbf F}}
	\def \bS{{\mathbf S}}
	\def \bC{{\mathbf C}}
	\def \bD{{\mathbf D}}
	\def \bfd{{\mathbf d}}
	\def \bfe{{\mathbf e}}
	\def \Arm{\mathrm{Arm}}

	\def \Ire{\I^{\rm re}}
	\def \Iim{\I^{\rm im}}
	\def \Iiso{\I^{\rm iso}}
	\def  \GL{\mathrm GL}
	\def \sgn{\rm sgn}

	\def \bs{\mathbf s}
	\def \bfx{\mathbf x}
	\def \I{\mathbb I}
	
	\def \bbK{{\mathbb K}}
	\def \Z{{\mathbb Z}}
	\def \F{{\mathbb F}}
	\def \C{{\mathbb C}}
	\def \N{{\mathbb N}}
	\def \Q{{\mathbb Q}}
	\def \G{{\mathbb G}}
	\def \X{{\mathbb X}}
	\def \P{{\mathbb P}}
	\def \K{{\mathbb K}}
	\def \E{{\mathbb E}}
	\def \A{{\mathbb A}}
	\def \BH{{\mathbb H}}
	\def \T{{\mathbb T}}
	
	\def \mod{{\mathrm{mod}}}
	
	\newcommand{\bluetext}[1]{\textcolor{blue}{#1}}
	\newcommand{\redtext}[1]{\textcolor{red}{#1}}
	\newcommand{\red}[1]{\redtext{ #1}}
	\newcommand{\blue}[1]{\bluetext{ #1}}
	\def \tMH{{\cs\cd\widetilde{\ch}}}
	
	\title[Growth rates of indecomposable summands]{Growth rates of indecomposable summands in tensor powers of  representations of quivers}

	\author[Ming Lu]{Ming Lu}
	\address{Department of Mathematics, Sichuan University, Chengdu 610064, P.R.China}
	\email{luming@scu.edu.cn}

	\author[Yayun Zhang]{Yayun Zhang}
	\address{Department of Mathematics, Sichuan University, Chengdu 610064, P.R.China}
	\email{2020141210001@stu.scu.edu.cn}

	\subjclass[2020]{Primary 16G20, 18M05,
		16E60.}
	\keywords{quiver representations; tensor product; indecomposable modules; tensor categories; Schur--Weyl duality}
	
	\begin{abstract}
		Tensor products of quiver representations have been extensively studied; typical examples include the pointwise tensor product and the tensor product induced by the coalgebra structure of path algebras. In this paper, we investigate the growth rates of the number of indecomposable direct summands in tensor powers of quiver representations with respect to these two typical tensor products.
	\end{abstract}
	
	\maketitle
	\setcounter{tocdepth}{1}
	\tableofcontents

	\section{Introduction}
	
	Tensor product decompositions of representations (or, more generally, objects of a tensor category) form a central topic in representation theory.
	In recent years, the asymptotic behavior of the number of indecomposable direct summands in tensor powers has attracted considerable attention, yet rather little is known; see \cite{B20,BS20, CEO23,CEO24, COT24,EK23}.
	In a remarkable paper, Coulembier, Ostrik, and Tubbenhauer \cite{COT24} investigated the exponential growth rates of indecomposable summands in tensor powers for representations of groups, quantum groups, and related monoidal categories, and proved that the growth rate coincides with the dimension of the underlying representation.
	
	Parallel to representations of groups and Hopf algebras, the representation theory of quivers provides a foundational framework for studying finite-dimensional algebras and their module categories.  The more general subject of bialgebra structures on quiver algebras has been studied in \cite{CR97,CR02,HLY05}. In particular, Cibils and Rosso classify in \cite{CR97} those quivers whose corresponding quiver algebras can be equipped with
	the structure of Hopf algebras.
	
	Tensor products of representations of quivers, can be dated back to \cite{D93}  by Dieterich for the $n$-subspace quivers in connection with his
	investigation of lattices over curve singularities, have been systematically studied by Herschend \cite{H08A,H08B,H09}, who introduced two fundamental monoidal structures on the category of quiver representations:
	the pointwise tensor product, defined vertex-wise and arrow-wise, and the tensor product induced by coassociative partitioning morphisms on path algebras.
	The former is elementary and concrete, while the latter arises from a coalgebra structure and thus fits into the classical framework of monoidal categories.

	The goal of the present paper is to study the exponential growth rates of the number of indecomposable summands in tensor powers of quiver representations with respect to the above two tensor products.
	We extend the philosophy and methods of \cite{COT24} to the setting of quiver representations and obtain explicit and sharp growth rate formulas for arbitrary quivers.
	
	The main results of this paper can be summarized in the following theorems.
	
	Let $Q=(Q_0,Q_1,s,t)$ be a quiver, and $M=(M_i,M_\alpha)_{i\in Q_0,\alpha\in Q_1}$ be a finite-dimensional representation of $Q$ over an algebraically closed field $\bfk$.  The pointwise tensor product of representations is defined in \cite[Section~2]{H08A}: for any representations $M=(M_i,M_\alpha),N=(N_i,N_\alpha)$ of $Q$, the \emph{pointwise tensor product} of $M,N$ is the representation $M\otimes N=(L_i,L_\alpha)$ defined as
	\[
	L_i=M_i\otimes N_i,\quad L_\alpha=M_\alpha\otimes N_\alpha,\quad\forall i\in Q_0,\alpha\in Q_1.
	\]
	
	Following \cite{COT24}, we define $b_n(M)$ as the number of indecomposable summands in $M^{\otimes n}$, for any $n\in\N$; and     
	$$\beta(M)=\lim\limits_{n\to +\infty}\sqrt[n]{b_n(M)},$$
	which is well defined by using Fekete's Subadditive Lemma~\cite{Fe23}. Then we have the following.

	\begin{customthm}{\bf A}[Theorem~\ref{thm:main0}]
		\label{thm:pointwise}
		Let \(Q\) be an arbitrary quiver and \(M\) a finite-dimensional representation of \(Q\).
		For the pointwise tensor product, we have 
		\[
		\beta(M) = \max_{i\in Q_0} \dim M_i,
		\]
		for any representation $M$ of $Q$.
	\end{customthm}
	
	Its proof is motivated by the approach in \cite[Theorem~1.9]{COT24}. We use Schur-Weyl duality between symmetric groups and products of general linear groups to reduce the argument to analyzing the exponential growth rate of the number of indecomposable summands in tensor powers of $M_i$ over the general linear group, a result already established in \cite[Proposition 2.2]{COT24}.

	Let $\Delta:\bfk Q\rightarrow \bfk Q\otimes \bfk Q$ be an algebra morphism. It is called a {\em partitioning morphism} \cite{H08A} if $\Delta$ is both an algebra morphism and a partitioning map; see Definition \ref{dfn:partition-morph}. The partitioning morphism induces a tensor product of representations, denoted by $\otimes^\Delta$. For any representation $M$, denote by $b_n^\Delta(M)$, $\beta^\Delta(M)$ the corresponding quantities.
	
	\begin{customthm}{\bf B}[Theorem~\ref{thm:main1}]
		\label{thm:partition}
		Let \(Q\) be an arbitrary quiver and \(M\) a finite-dimensional representation of \(Q\).
		For the tensor product induced by a coassociative partitioning morphism with the property \eqref{eq:diag}, we have
		\[
		\beta^\Delta(M) = \dim M,
		\]  
		for any representation $M$ of $Q$.
	\end{customthm}
	
	The proof is to compare $b_n^\Delta(M)$ and $b_n(M)$, which is given in Corollary~\ref{cor:RelationBetweenb&bDelta}:
	\begin{align*}
		b^\Delta_n(M)=b_n(M)+(\dim M)^n-\sum\limits_{k\in Q_0} (\dim M_k)^n.
	\end{align*}
	In particular, $b_n^\Delta(M)$ is independent of the partitioning morphisms $\Delta$.

	However, it is very difficult to compute $b_n(M)$ (or equivalently $b_n^\Delta(M)$) explicitly. As examples, we obtain explicit closed-form formulas for $b_n(M)$,
	in terms of dimension vectors for arbitrary quivers $Q$ of type $A$ or $D$; see Theorem~\ref{thm:bForTypeD} for type $D$, and see Theorem~\ref{thm:bForTypeA} for type $A$.
	From the computations of type $D$, we find that $b_n(M)$ also depends on the orientation of the quiver $Q$; see also \cite{H09}.

	The paper is organized as follows.
	In Section~\ref{sec:prelim}, we recall basic definitions and preliminary results on quiver representations, tensor products, and monoidal categories.
	Section~\ref{sec:tensor1} is devoted to the pointwise tensor product and the proof of Theorem~\ref{thm:pointwise}.
	In Section~\ref{sec:tensor2}, we study the tensor product via partitioning morphisms and prove Theorem~\ref{thm:partition}.
	Section~\ref{sec:example} presents explicit combinatorial formulas for Dynkin quivers of type \(A\) and \(D\). Since the proofs for type D rely on lengthy computations, we include the technical details in the appendix. 
	
	\vspace{2mm}
	\noindent{\bf Acknowledgments} 
	ML is partially supported by the National Natural Science Foundation of China (No. 12171333). 
	
	\section{Preliminaries}
	
	\label{sec:prelim}
	
	Throughout this paper, let $\mathbf{k}$ be an algebraically closed field. 
	
	\subsection{Monoidal categories}
	\begin{definition}[\cite{EGNO15}, Definitions 2.1.1, 2.1.2]
		A \emph{monoidal category} is a quintuple $(\mathcal{C},\otimes ,a,\mathbf{1},\imath)$ where $\mathcal{C}$ is a category, $\otimes :\mathcal{C}\times \mathcal{C}\to \mathcal{C}$ is a bifunctor called the \emph{tensor product} bifunctor, $a:(-\otimes -)\otimes -\to -\otimes (-\otimes -)$ is a natural isomorphism:
		\[a_{X,Y,Z}:(X\otimes Y)\otimes Z\to X\otimes (Y\otimes Z),\quad X,Y,Z\in \mathcal{C}\]
		called the \emph{associativity constraint} (or \emph{associativity isomorphism}), $\mathbf{1}\in \mathcal{C}$ is an object of $\mathcal{C}$, and $\imath:\mathbf{1}\otimes \mathbf{1}\to \mathbf{1}$ is an isomorphism, subject to the following two axioms.
		\begin{enumerate}
			\item{\rm (The pentagon axiom)} The equations \[(\id_W\otimes a_{X,Y,Z})(a_{W,X\otimes Y,Z})(a_{W,X,Y}\otimes \id_Z)=a_{W,X,Y\otimes Z} a_{W\otimes X,Y,Z}\] hold for all objects $W,X,Y,Z$ in $\mathcal{C}$;
			\item \rm{(The unit axiom)} The functors 
			\[
			L_{\mathbf{1}}:X\mapsto \mathbf{1}\otimes X \quad \text{and} \quad R_\mathbf{1}:X\mapsto X\otimes \mathbf{1}
			\]
			of left and right multiplication by $\mathbf{1}$ are autoequivalences of $\mathcal{C}$.
		\end{enumerate}
		The object $\mathbf{1}$ is called the \emph{unit object} of $\mathcal{C}$.
	\end{definition}
	
	\begin{definition}[\cite{EGNO15}, Definition 2.4.1]
		Let $(\mathcal{C},\otimes,\mathbf{1},a,\imath)$ and $(\mathcal{C}',\otimes',\mathbf{1}',a',\imath')$ be two monoidal categories. A \emph{monoidal functor} from $\mathcal{C}$ to $\mathcal{C}'$ is a pair $(F,J)$ , where $F:\mathcal{C}\to \mathcal{C}'$ is a functor, and
		\[J_{X,Y}:F(X)\otimes' F(Y)\to F(X\otimes Y)\]
		is a natural isomorphism, such that $F(\mathbf{1})$ is isomorphic to $\mathbf{1}'$ and the equation
		\begin{equation}\label{eq:TheMonoidalStructureAxiom}
			F(a_{X,Y,Z})J_{X\otimes Y,Z}(J_{X,Y}\otimes'\id_{F(Z)})=J_{X,Y\otimes Z}(\id_{F(X)}\otimes' J_{Y,Z})a'_{F(X),F(Y),F(Z)}
		\end{equation}
		holds for all $X,Y,Z\in \mathcal{C}$.
	\end{definition}
	
	\subsection{Tensor products representations of general linear groups}
	Let $G$ be a group. The diagonal map
	\[
	\mathbf{k} G \to \mathbf{k} G \otimes \mathbf{k} G, \qquad g \mapsto g \otimes g,
	\]
	makes the group algebra $\mathbf{k} G$ a coalgebra; consequently, the category of left (equivalently, right) $G$-modules becomes a monoidal category, where the tensor product is equipped with the diagonal action
	\[
	g \cdot (m \otimes n) = (g \cdot m) \otimes (g \cdot n)
	\]
	for $G$-modules $M, N$ and for all $g \in G$, $m \in M$, $n \in N$.
	
	For a representation $M$ of $G$ and a positive integer $n$, we denote by $M^{\otimes n}$ the $n$-fold tensor product of $M$ and by $nM$ the $n$-fold direct sum of $M$.
	For a finite set $S$, let $\sharp S$ denote its cardinality.
	
	\begin{definition}[{\cite[Definition 1.1]{COT24}}]\label{def:GrowthRateOfRepOfGroup}
		For a representation $M$ of $G$ and a positive integer $n$, define
		\[
		b_n^G(M) := \#\{\text{indecomposable summands in } M^{\otimes n} \text{ counted with multiplicities}\}.
		\]
		Furthermore, set
		\[
		\beta^G(M) = \lim_{n \to +\infty} \sqrt[n]{b_n^G(M)}.
		\]
	\end{definition}
	
	Let $V$ be a $d$-dimensional $\mathbf{k}$-vector space. The general linear group $\GL_d(\mathbf{k})$ acts on $V$ by matrix multiplication; with this action, $V$ is called the \emph{natural representation} of $\GL_d(\mathbf{k})$.
	
	\begin{lemma}[{\cite[Proposition 2.2]{COT24}}]\label{lem:growth-GLd}
		Let $V$ be the natural representation of $\GL_d(\bfk)$. We have $\beta^{\GL_d(\mathbf{k})}(V) = \dim V$.
	\end{lemma}
	
	\subsection{Schur-Weyl duality}
	\label{subsec:Schur-Weyl}
	Let $\bfk$ be an algebraically closed field of characteristic $p=0$ or a prime. 
	Let $n$ be a positive integer and $\mathfrak{S}_n$ the symmetric group on $n$ letters. 
	Denote by $\Irr_\bfk(\mathfrak{S}_n)$ the set of irreducible (right) $\mathfrak{S}_n$-modules over $\bfk$. To describe the objects of $\Irr_\bfk(\mathfrak{S}_n)$, we need the notion of $p$-regular partitions.
	A \emph{partition} $\lambda$ of $n$, written $\lambda \vdash n$, is a finite tuple $(\lambda_1, \lambda_2, \dots, \lambda_r)$ of positive integers such that $\lambda_1 \ge \lambda_2 \ge \cdots \ge \lambda_r$ and $\sum\limits_{i=1}^r \lambda_i = n$. The number $r$ is called the \emph{length} of $\lambda$ and is denoted by $\ell(\lambda)$.
	
	\begin{definition}[{\cite[Page 241]{JK81}}]
		A partition $\lambda=(\lambda_1, \lambda_2, \dots, \lambda_d)$ of $n$ is called \emph{$p$-singular} if there exists $i$ with 
		\[\lambda_{i+1}=\lambda_{i+2}=\cdots=\lambda_{i+p}>0;\]
		otherwise $\lambda$ is called \emph{$p$-regular}. 
	\end{definition}
	There is a bijection 
	\[\Irr_\bfk(\mathfrak{S}_n) \stackrel{1:1}{\longleftrightarrow} \{\lambda\vdash n\mid \lambda \text{ is }p\text{-regular}\},\] 
	see e.g. \cite[Corollary~6.1.12]{JK81}. 
	
	For a $p$-regular partition $\lambda\vdash n$, denote by $D_\lambda$ the corresponding irreducible $\mathfrak{S}_n$-module. 
	Then we obtain a primitive  orthogonal idempotent decomposition
	
	\begin{align}
		\label{eq:idem-decom}
		1=\sum_{\lambda: p\text{-regular}}\Big(e_{\lambda,1}+\cdots+e_{\lambda,d_\lambda}\Big),
	\end{align}
	where $d_\lambda:=\dim_\bfk D_\lambda$ for each $p$-regular partition $\lambda$.
	In other words, a primitive orthogonal idempotent decomposition of the identity consists of $d_\lambda=\dim D_\lambda$ primitive idempotents corresponding to $D_\lambda$, for each $p$-regular partition $\lambda$.
	
	Let $V$ be the natural representation of the general linear group $\GL_d(\bfk)$. 
	For any positive integer $n$, the $n$-fold tensor space $V^{\otimes n}$ carries a $(\GL_d(\bfk),\mathfrak{S}_n)$-bimodule structure given by
	\begin{align*}
		g\cdot(v_1\otimes\cdots\otimes v_n)&=g(v_1)\otimes\cdots\otimes g(v_n),\\
		(v_1\otimes\cdots\otimes v_n)\cdot \sigma&=v_{\sigma(1)}\otimes \cdots \otimes v_{\sigma(n)},\label{eq:S_nModuleStructureOverTensorProduct}
	\end{align*}
	for $v_i\in V,1\leq i\leq n$,$g\in \GL_d(\bfk),\sigma\in \mathfrak{S}_n$.
	
	\begin{lemma}[Schur-Weyl duality; see e.g. {\cite[Theorem~4.1]{DP76}}]\label{lem:SchurWeylDuality-Epimorphism}
		We have two epimorphisms:
		\begin{align*}
			\bfk\mathfrak{S}_n\twoheadrightarrow \End_{\GL_d(\bfk)}(V^{\otimes n}), \qquad \bfk\GL_d(\bfk)\twoheadrightarrow\End_{\mathfrak{S}_n}(V^{\otimes n}). 
		\end{align*}  
	\end{lemma}
	
	When $p=0$, as a $\GL_d(\bfk)$-module, the $\GL_d(\mathbf{k})$-module $V^{\otimes n}$ decomposes as
	\begin{equation}
		V^{\otimes n}\cong \bigoplus_{\lambda\vdash n,\ell(\lambda)\leq d}L(\lambda)\otimes_\bfk D_\lambda,\label{eq:SchurWeylDuality-CorrespondenceBetweenSimpleModule}
	\end{equation}
	where $L(\lambda)=\Hom_{\mathfrak{S}_n}(D_\lambda,V^{\otimes n})\cong V^{\otimes n}\cdot e_\lambda$ is the irreducible $\GL_d(\bfk)$-module. This 
	is known as the classical \emph{Schur-Weyl duality}.
	
	The decomposition \eqref{eq:SchurWeylDuality-CorrespondenceBetweenSimpleModule} does not extend to arbitrary characteristic. Nevertheless, in general characteristic the epimorphism from Lemma \ref{lem:SchurWeylDuality-Epimorphism},
	\[
	\bfk\mathfrak{S}_n\twoheadrightarrow \End_{\GL_d(\bfk)}(V^{\otimes n})
	\]
	still implies that the indecomposable direct summands of $V^{\otimes n}$ as a $\GL_d(\bfk)$-module are governed by the action of $\bfk\mathfrak{S}_n$. More precisely, the number of indecomposable summands of $V^{\otimes n}$ equals the number of primitive idempotents in \eqref{eq:idem-decom} whose images in $\End_{\GL_d(\bfk)}(V^{\otimes n})$ act nontrivially on $V^{\otimes n}$.
	Equivalently, fixing a primitive idempotent decomposition \eqref{eq:idem-decom}, we have
	\begin{equation}\label{eq:SchurWeylDuality-Decomposition}
		V^{\otimes n} \cong \bigoplus_{\lambda: p\text{-regular}} \ \bigoplus_{i=1}^{d_\lambda} V^{\otimes n} \cdot e_{\lambda,i},
	\end{equation}
	where each nonzero summand $V^{\otimes n}\cdot e_{\lambda,i}$ is indecomposable. 
	In particular, only those $p$-regular partitions $\lambda$ with $\ell(\lambda)\le d$ contribute nontrivially, so the number of indecomposable direct summands is given by
	\[
	\sum_{\substack{\lambda:p-\text{regular} \\ \ell(\lambda)\le d}} \dim D_\lambda.
	\]

	\subsection{Quiver representations}
	\label{subsec:quiver}
	Let $Q=(Q_0,Q_1,s,t)$ be an arbitrary quiver, where $Q_0$ is the set of \emph{vertices}, $Q_1$ is the set of \emph{arrows}, $s,t:Q_1\rightarrow Q_0$ are the \emph{source and target maps}, respectively.
	Denote by $\bfk Q$ the path algebra. 
	A \emph{representation} of $Q$ (over $\bfk$) is a family $M=(M_i,M_\alpha)_{i\in Q_0,\alpha\in Q_1}$, where each $M_i$ is a $\bfk$-vector space and each $M_\alpha:M_i\rightarrow M_j$ is a $\bfk$-linear map. We identify representations of $Q$ with left $\bfk Q$-modules. 
	A representation $M$ is called \emph{finite-dimensional} if $\dim M<\infty$. We denote by $\underline{\dim} M=(\dim M_i)_{i\in Q_0}$ the dimension vector of $M$. In this paper, we only consider finite-dimensional representations.
	
	Let $M,N$ be two representations of a quiver $Q$. A \emph{homomorphism} $f:M\to N$ is a family $f=(f_i)_{i\in Q_0}$ of linear maps, where $f_i$ are linear maps from $M_i$ to $N_i$, such that for each arrow $\alpha\in Q_1$, we have $f_{t(\alpha)}M_\alpha=N_\alpha f_{s(\alpha)}$. Let $\rep_\bfk (Q)$ be the category of finite-dimensional representations of $Q$, which is equivalent to the category $\mod(\bfk Q)$ of finite-dimensional left $\bfk Q$-modules. 
	We denote by $\Hom_Q(M,N)$ the set of homomorphisms from $M$ to $N$ in the category $\rep_\bfk (Q)$.
	
	If there exists an isomorphism $f:M\to N$, we say that the representations $M$ and $N$ are \emph{isomorphic} and denote the fact by $M\cong N$.
	For a representation $M$, we denote by $\id_M$ the endomorphism of $M$ defined as $(\id_M)_i=\id_{M_i}$, where $\id_{M_i}(v)=v$, for each vertex $i\in Q_0$ and vector $v\in M_i$.
	
	\section{Pointwise tensor products of quiver representations}
	
	\label{sec:tensor1}
	
	In this section, we will review the construction of pointwise tensor products of representations in \cite[Section~2]{H08A} and state the main result following \cite{COT24}.
	
	\subsection{Tensor products of quiver representations}
	\label{subsec:tensor1}
	Let $Q=(Q_0,Q_1,s,t)$ be an arbitrary quiver.
	\begin{definition}[{\cite[Proposition 2]{H08A}}]	\label{def:DiagTensorProductOfRepsOfQuiver}
		For any representations $M=(M_i,M_\alpha),N=(N_i,N_\alpha)$ of $Q$, the \emph{pointwise tensor product} of $M,N$ is the representation $M\otimes N=(L_i,L_\alpha)$ defined as
		\[
		L_i=M_i\otimes N_i,\quad L_\alpha=M_\alpha\otimes N_\alpha,\quad\forall i\in Q_0,\alpha\in Q_1.
		\]
		For homomorphisms $f=(f_\alpha):M\to M'$, $g=(g_\alpha):N\to N'$ of representations of $Q$, the \emph{tensor product} of $f$ and $g$ is the homomorphism $f\otimes g:M\otimes N\to M'\otimes N'$, defined as
		\[
		(f\otimes g)_\alpha=f_\alpha\otimes g_\alpha.
		\]
		
	\end{definition}
	
	Recall from the representation theory of algebras that every finite-dimensional representation can be uniquely decomposed into a direct sum of indecomposable representations up to isomorphism.
	
	Let us consider $\rep_\bfk (Q)$ with the tensor product $\otimes$. For any representations $L,M,N$ of $Q$, we have the following canonical vector space isomorphisms
	\begin{align*}
		& (L\otimes M)\otimes N\cong L\otimes (M\otimes N),
		\\
		&L\otimes M\cong M\otimes L,
		\\
		&L\otimes (M\oplus N)\cong (L\otimes M)\oplus (L\otimes N).
	\end{align*}
	From \cite[Theorem~2]{H08A}, we know that they are also isomorphisms as representations of $Q$, thereby endowing the category $\rep_\bfk (Q)$ with the structure of a monoidal category in which the tensor product distributes over the direct sum. The unit object $\mathbf{1}=(\mathbf{1}_i,\mathbf{1}_\alpha)$ is defined as
	\begin{align*}
		\mathbf{1}_i=\bfk,\qquad \mathbf{1}_\alpha=\id,\quad \forall i\in Q_0,\alpha\in Q_1.
	\end{align*}
	
	Let $f:M\to N$ and $g:M'\to N'$ be two isomorphisms of representations of $Q$. Denote by $f^{-1}$ and $g^{-1}$ the inverse of $f$ and $g$, then $(f\otimes g)^{-1}=f^{-1}\otimes g^{-1}$. Hence if $M\cong M'$ and $N\cong N'$, then $M\otimes N\cong M'\otimes N'$.
	
	For a representation $M$ of $Q$ and an positive integer $n$, we denote by $M^{\otimes n}$ the $n$-fold tensor product of $M$, $nM$ the $n$-fold direct sum of $M$.
	
	\subsection{Growth rates of the number of indecomposable summands in tensor products}
	
	The following definition is inspired by the notion for group representations in \cite[Definition~1.1]{COT24}, which we have reproduced in Definition~\ref{def:GrowthRateOfRepOfGroup}.

	\begin{definition}
		For a representation $M$ of quiver $Q$ and a positive integer $n$, 
		we define
		\begin{align*}
			b^Q_n(M):=\sharp\{\text{indecomposable summands in }M^{\otimes n}\text{ counted with multiplicities}\}. 
		\end{align*}
		Let further
		\begin{align}
			\beta^Q(M)=\lim\limits_{n\to +\infty}\sqrt[n]{b^Q_n(M)}.
		\end{align}
	\end{definition}
	When the quiver $Q$ is clear from the context, we simply write $b_n(M)$ and $\beta(M)$ for $b_n^Q(M)$ and $\beta^Q(M)$, respectively.
	By using the distributive law of tensors and direct sums, we have  
	$$b_n(M\otimes N)\geq b_n(M)b_n(N),\quad \forall n\in\N,$$ 
	for any representations $M$, $N$. In particular, $b_m(M)b_n(M)\geq b_{m+n}(M)$, so $\beta(M)$ is well-defined by Fekete’s Subadditive Lemma.
	
	Now we can state the first main result of this paper.
	
	\begin{theorem}
		\label{thm:main0}
		Let $Q$ be an arbitrary quiver. Then we have 
		\[
		\beta(M) = \max_{i\in Q_0} \dim M_i
		\]
		for any representation $M$ of $Q$.
	\end{theorem}
	The proof is presented in Subsection~\ref{sec:proof}.
	The following Lemma~is an easy observation. 
	
	\begin{lemma}
		\label{lem:upper}
		Let $Q$ be an arbitrary quiver. Then we have \[
		\beta(M)\leq \max_{i\in Q_0} \dim M_i,
		\]
		for any representation $M$ of $Q$.
	\end{lemma}
	
	\begin{proof}
		For $M=(M_i,M_\alpha)$, set $d_i=\dim M_i$ for $i\in Q_0$.
		By definition,
		\[b_n(M)\leq \dim (M^{\otimes n})=\sum_{i\in Q_0}d_i^n.\]
		The desired result then follows from Lemma~\ref{sqrtnlimOfSum}~(1) below.
	\end{proof}
	
	Thus, the proof of Theorem~\ref{thm:main0} reduces to showing $\beta(M)\geq \max_{i\in Q_0} \dim M_i$, which is established in Subsection~\ref{sec:proof}. The argument relies on Schur–Weyl duality, recalled in Subsection~\ref{subsec:Schur-Weyl}.
	
	We require the following elementary facts from calculus.
	\begin{lemma}\label{sqrtnlimOfSum}
		Let $\{a^{(i)}_n\mid n=1,2,\cdots\}\subseteq \N,1\leq i\leq m$. Then the following statements hold.
		\begin{enumerate}
			\item Suppose that 
			\[\lim\limits_{n\to+\infty}\sqrt[n]{a^{(i)}_n}=A_i,\quad \forall 1\leq i\leq m.\]
			Then $\lim\limits_{n\to +\infty}\sqrt[n]{\sum\limits_{i=1}^m a^{(i)}_n}=\max\{A_i\mid 1\leq i\leq m\}$;\\
			\item Let $\{a_n\mid n=1,2,\cdots\}$ be a positive sequence with $\lim\limits_{n\to +\infty}a_n=1$. Then $\lim\limits_{n\to +\infty}\sqrt[n]{a_n}=1$;
			\item The following limit holds:
			\[\lim\limits_{n\to +\infty}\sqrt[n]{(\sum\limits_{i=1}^ma_n^{(i)})^n-\sum\limits_{i=1}^m(a_n^{(i)})^n}=\lim\limits_{n\to +\infty}\sum\limits_{i=1}^ma_n^{(i)}.\]
		\end{enumerate}
	\end{lemma}
	\begin{proof}
		These results are standard; we include a proof for completeness.
		
		(1)
		Let $M_n=\max\{a_n^{(i)}\mid 1\leq i\leq m\}$. Then $M_n\leq \sum\limits_{i=1}^n a_n^{(i)}\leq mM_n$.
		It is straightforward to verify that $\lim\limits_{n\to +\infty}\sqrt[n]{M_n}=\max\{A_i\mid 1\leq i\leq m\}$.
		
		(2)
		Since $\lim\limits_{n\to +\infty}a_n=1$, there exists $N\in \N$ such that $\frac{1}{2}\leq a_n\leq \frac{3}{2}$ for all $n\geq N$. Consequently, $\sqrt[n]{\frac{1}{2}}\leq \sqrt[n]{a_n}\leq \sqrt[n]{\frac{3}{2}}$ for $n\geq N$. 
		
		Both (1) and (2) now follow from the squeeze theorem and the basic limit $\lim\limits_{n\to +\infty}\sqrt[n]{c}=1$ for any constant $c>0$.
		
		(3)
		Write $S_n=\sum\limits_{i=1}^m a_n^{(i)}$. Then $(\sum\limits_{i=1}^ma_n^{(i)})^n-\sum\limits_{i=1}^m(a_n^{(i)})^n= (S_n)^n(1-\sum\limits_{i=1}^m(\frac{a_n^{(i)}}{S_n})^n)$. 
		Taking the $n$-th root and applying (2) to $1-\sum\limits_{i=1}^m(\frac{a_n^{(i)}}{S_n})^n$, we obtain the desired result.
	\end{proof}
	
	\subsection{Proof of Theorem~\ref{thm:main0}}
	\label{sec:proof}
	Now let us complete the proof of Theorem~\ref{thm:main0}. 
	The proof is inspired by \cite[Theorem~1.9~(a)]{COT24}.
	
	\begin{proof}[Proof of Theorem~\ref{thm:main0}]
		For a representation $M=(M_i,M_\alpha)$ of a  quiver $Q$, we denote $d_i=\dim M_i$ for $i\in Q_0$.
		Consider the tensor product $M^{\otimes n}$. 
		The symmetric group $\mathfrak{S}_n$ acts naturally on $M^{\otimes n}$ by
		\[
		\rho_\sigma:M^{\otimes n}\to M^{\otimes n},\quad m_1\otimes \cdots \otimes m_n\mapsto m_{\sigma(1)}\otimes \cdots\otimes m_{\sigma(n)},\quad \forall \sigma\in\mathfrak{S}_n. 
		\]
		One checks that $\rho_\sigma\in \End_{\bfk Q}(M^{\otimes n})$ and that the map $\rho:\bfk \mathfrak{S}_n\to \End_{\bfk Q}(M^{\otimes n})$, $\sigma\mapsto \rho_\sigma$, is an algebra anti-homomorphism; thus $M^{\otimes n}$ is a $(\bfk Q,\mathfrak{S}_n)$-bimodule.
		Hence for each $e\in \bfk\mathfrak{S}$, $M^{\otimes n}\cdot e$ is a left $\bfk Q$-submodule of $M^{\otimes n}$.
		In particular, let $1=\sum\limits_{\lambda:p\text{-regular}}e_{\lambda}$ be an orthogonal idempotent decomposition of the identity in $\bfk \mathfrak{S}_n$; see \eqref{eq:idem-decom}. 
		Then for any $p$-regular partition $\lambda$, $M^{\otimes n}. e_{\lambda}$ is a left $\bfk Q$-submodule of $M^{\otimes n}$.
		Then 
		$$M^{\otimes n}=\bigoplus_{\lambda:p\text{-regular}} M^{\otimes n}\cdot e_{\lambda}$$ 
		is a direct sum decomposition as left $\bfk Q$-modules. 
		Since $M^{\otimes n}\cdot e_\lambda$ may not be indecomposable, we have the inequality  
		\[b_n(M)\geq \sum\limits_{\lambda:M^{\otimes n}\cdot e_\lambda\neq 0}\sharp \{\text{indecomposable summands in }M^{\otimes n}\cdot e_\lambda\text{ counted with multiplicities}\}.
		\]
		Moreover, $e_\lambda= e_{\lambda,1}+\cdots+e_{\lambda,d_\lambda}$ is a sum of primitive idempotents $e_{\lambda,j}$ ($1\leq j\leq d_\lambda$) corresponding to $D_\lambda$. 
		Thus we further have
		\begin{align}
			\label{eq:growth-quiver}
			b_n(M)\geq \sum\limits_{\lambda:M^{\otimes n}\cdot e_\lambda\neq 0}\dim D_\lambda.
		\end{align}
		
		There is also a natural action $\eta$ of the direct product $\prod_{i\in Q_0}\GL_{d_i}(\bfk)$ on $M^{\otimes n}$:
		\[(\eta_g)_i:(M_i)^{\otimes n}\to (M_i)^{\otimes n}, \quad m_1\otimes \cdots\otimes m_n\mapsto g_i(m_1)\otimes \cdots \otimes g_i(m_n),\]
		where $g=(g_i\mid i\in Q_0),g_i\in \GL_{d_i}(\bfk)$.
		By \eqref{eq:SchurWeylDuality-Decomposition}, for each $M_i$ viewed as a $(\GL_{d_i}(\bfk),\mathfrak{S}_n)$-bimodule, we have 
		\begin{align}
			\label{eq:growth-Schur}
			b_n^{\GL_{d_i}(\bfk)}(M_i)=\sum\limits_{\lambda:(M_i)^{\otimes n}\cdot e_\lambda\neq 0}\dim D_\lambda;
		\end{align}
		see the proof of \cite[Theorem~1.9~(a)]{COT24}. 
		
		As vector spaces, $M^{\otimes n}\cdot e_\lambda=\oplus_{i\in Q_0}M_i^{\otimes n}\cdot e_\lambda$. Hence $M^{\otimes n}\cdot e_\lambda\neq 0$ if and only if $(M_i)^{\otimes n}\cdot e_\lambda\neq 0$ for some $i\in Q_0$.
		Consequently,
		\begin{align}
			\label{eq:lambda-Q-spaces}
			\sharp Q_0 \cdot \sharp \{\lambda\mid M^{\otimes n}\cdot e_\lambda\neq 0\}\geq \sum\limits_{i\in Q_0}\sharp \{\lambda\mid (M_i)^{\otimes n}\cdot e_\lambda\neq 0\}.
		\end{align}
		Combining \eqref{eq:growth-quiver}, \eqref{eq:growth-Schur}, and \eqref{eq:lambda-Q-spaces},
		we obtain
		\begin{align}
			\sharp  Q_0 \cdot b_n(M)\geq & \sharp  Q_0 \cdot \sum\limits_{\lambda:M^{\otimes n}\cdot e_\lambda\neq 0}\dim D_\lambda =\sharp  Q_0\cdot \sum\limits_{\lambda:\exists i\in Q_0,M_i^{\otimes n}\cdot e_\lambda\neq 0}  \dim D_\lambda \\\notag
			\geq &\sum\limits_{i\in Q_0}\sum\limits_{\lambda:(M_i)^{\otimes n}\cdot e_\lambda\neq 0}\dim D_\lambda =\sum\limits_{i\in Q_0}b_n^{\GL_{d_i}(\bfk)}(M_i).
		\end{align}
		Applying Lemma~\ref{sqrtnlimOfSum}~(1)(2), Lemma~\ref{lem:growth-GLd}, and the order-preserving property of limits, we deduce $\beta(M)\geq \max_{i\in Q_0} d_i$. Together with Lemma~\ref{lem:upper}, we complete the proof.
	\end{proof}

	\section{Tensor products induced by the coalgebra structure}
	
	\label{sec:tensor2}

	It is remarkable that the tensor product of representations on quivers $Q$ given in Subsection~\ref{subsec:tensor1} does not arise from any coalgebra structure on $\bfk Q$. However, another variant of the tensor product, which does come from a coproduct on the path algebra, is defined in \cite{H08A} using partitioning morphisms; see Subsection~\ref{subsec:partitioning}. In this section, we shall study the growth rate of tensor powers for this kind of tensor product. 
	
	\subsection{Partitioning morphisms}
	\label{subsec:partitioning}
	
	\begin{definition}[\cite{H08A}]
		Let $\Delta:\bfk Q\to \bfk Q\otimes_\bfk \bfk Q$ be an algebra morphism, and let $M,N$ be representations of $Q$. 
		The \emph{tensor product of $M,N$ respect to $\Delta$} is defined as $M\otimes^\Delta N$, with the action
		\[
		\alpha\cdot (m\otimes n)=\sum\limits_{\mu,\nu} \lambda_{\mu,\nu}(\mu\cdot m)\otimes(\nu\cdot n),
		\]
		where $\alpha\in \bfk Q$ and $\Delta(\alpha)=\sum\limits_{\mu,\nu}\lambda_{\mu,\nu}\mu\otimes \nu$.
	\end{definition}
	
	The tensor product $\otimes^\Delta$ is associative if and only if $\Delta$ is coassociative, i.e., $(1 \otimes \Delta)\Delta = (\Delta \otimes 1)\Delta$. Under this condition, the notions introduced in Section~\ref{sec:tensor1} extend naturally to $\otimes^\Delta$, and we denote the corresponding objects by $M^{\otimes^\Delta n}$, $b_n^\Delta(M)$ and $\beta^\Delta(M)$.
	
	We denote by $\bigsqcup$ the disjoint union. A \emph{partition} of a set $E$ is a family $\{E_i\mid 1\leq i\leq n\}$ of pairwise disjoint subsets such that $\bigsqcup_{1\leq i\leq n}E_i=E$. 
	For a set $E$, we denote by $E^{\times n}$ the $n$-fold Cartesian product of $E$. 
	For each $i\in Q_0$, let $e_i$ denote the trivial paths of length $0$. The unit element of the path algebra $\bfk Q$ is then $1_{\bfk Q}=\sum_{i\in Q_0}e_i$. 
	
	\begin{definition}[\cite{H08A}]\label{dfn:partition-morph}
		Let $\Delta: \bfk Q\to \bfk Q\otimes_\bfk \bfk Q$ be a linear map.
		\begin{enumerate}
			\item $\Delta$ is called a \emph{partitioning map} if there exists a partition $\{E_k\mid k\in Q_0\}$ of the set $Q_0^{\times 2}$ satisfying \[\Delta(e_k)=\sum\limits_{(i,j)\in E_k}e_i\otimes e_j,\]
			for all vertices $k\in Q_0$;\\
			\item $\Delta$ is called a \emph{partitioning morphism} if $\Delta$ is both an  algebra morphism and a partitioning map.
		\end{enumerate}
	\end{definition}
	
	For a partitioning map $\Delta$, we denote by $\{E_k\mid k\in Q_0\}$ the corresponding partition.
	We only consider the partitioning maps $\Delta$ satisfying
	\begin{align}
		\label{eq:diag}
		\Delta(\mu)=\mu\otimes\mu, 
	\end{align}
	for all paths $\mu$ of length at least $1$. 
	Then by \cite[Proposition 4]{H08A}, $\Delta$ is a partitioning morphism if and only if $(k,k)\in E_k$, for all $k\in Q_0$. In this way, the partitioning morphisms $\Delta$  with the property \eqref{eq:diag} are uniquely determined by the partitions $\{E_k\mid k\in Q_0\}$ such that $(k,k)\in E_k$ for all $k\in Q_0$. 
	
	The main result of this section is the following.
	
	\begin{theorem}\label{thm:main1}
		Let $\Delta$ be a coassociative partitioning morphism with the property \eqref{eq:diag}. Then $\beta^\Delta(M)=\dim M$ for any representation $M$ of the quiver $Q$.
	\end{theorem}
	
	The proof, which is given in Subsection~\ref{sec:proof-1}, relies on a detailed study of the properties of the partitioning morphism.
	
	\begin{lemma}\label{EquivalentDescriptionOfCAThroughPartition}
		A partitioning morphism $\Delta$ with the property \eqref{eq:diag} is coassociative if and only if
		\begin{align*}
			&\{(a,b,c)\mid \exists  a'\in Q_0 \text{ such that }(a,a')\in E_i,(b,c)\in E_{a'}\}\\
			&=\{(a,b,c)\mid \exists c'\in Q_0 \text{ such that }(c',c)\in E_i,(a,b)\in E_{c'}\},
		\end{align*}
		for all $i\in Q_0$.
	\end{lemma}
	
	\begin{proof}
		The path algebra $\bfk Q$
		is generated by all $e_i\in Q_0$, $\alpha\in Q_1$.      Note that 
		\[
		(\Delta\otimes 1)\Delta(\alpha)=\alpha\otimes \alpha\otimes \alpha=(1\otimes \Delta)\Delta(\alpha), \quad\forall\alpha\in Q_1.
		\]
		Hence, $\Delta$ is coassociative if and only if $(\Delta\otimes 1)\Delta(e_i)=(1\otimes \Delta)\Delta(e_i)$, for all vertices $i\in Q_0$, 
		in other words, if and only if 
		\[
		\sum\limits_{(a,b,c):(c',c)\in E_i,(a,b)\in E_{c'}}e_a\otimes e_b\otimes e_c=\sum\limits_{(a,b,c):(a,a')\in E_i,(b,c)\in E_{a'}}e_a\otimes e_b\otimes e_c,\]
		for all $i\in Q_0$.
		The result now follows from the fact that $\{e_a\otimes e_b\otimes e_c\mid a,b,c\in Q_0\}$ is a linearly independent set in the $\bfk$-vector space $\bfk Q\otimes_\bfk \bfk Q\otimes_\bfk \bfk Q$.
	\end{proof}
	This observation motivates the following notations.
	Let $\Delta$ be a partitioning map and let $\{E_k\mid k\in I\}$ be the corresponding partition of the Cartesian square $E^{\times 2}$, where $I$ is an index set. 
	For $k\in I$, $n\geq 2$, we define
	\begin{align*}
		L_{n,k}^\Delta&:=\{(a_1,\cdots,a_n)\in E^{\times n}\mid \exists \;a_3',a_4',\dots,a_n', \text{ such that }(a_n',a_n)\in E_k,\\
		&\qquad \quad\qquad\qquad\qquad \qquad  (a_{n-1}',a_{n-1})\in E_{a_n'},\cdots,(a_1,a_2)\in E_{a_3'}\},\\
		R_{n,k}^\Delta&:=\{(a_1,\cdots,a_n)\in E^{\times n}\mid \exists\; a_1',a_2',\dots,a_{n-2}',\text{ such that } (a_1,a_1')\in E_k,
		\\
		&
		\qquad \quad\qquad\qquad\qquad \qquad(a_2,a_2')\in E_{a_1'},\cdots,(a_{n-1},a_n)\in E_{a_{n-2}'}\},
	\end{align*}
	and we adopt the convention that $L_{1,k}^\Delta=R_{1,k}^\Delta=\emptyset$. 
	
	One readily check that $L^\Delta_{2,k}=R^\Delta_{2,k}=E_k$ for every $k\in Q_0$. 
	It is remarkable that if, in addition, $\Delta$ is  assumed in addition to be an algebra morphism, then the tuple
	$$
	\overbrace{(k,k,\cdots,k)}^n
	$$
	belongs to $L_{n,k}^{\Delta}\cap R_{n,k}^\Delta$.
	When $\Delta$ is clear from the context, we simply write $L_{n,k}$ and $R_{n,k}$ for $L_{n,k}^\Delta$ and $R_{n,k}^\Delta$, respectively.
	
	Consequently, Lemma~\ref{EquivalentDescriptionOfCAThroughPartition} can be restated via the following equivalences.
	\begin{corollary}
		Let $\Delta$ be a partitioning morphism with the property \eqref{eq:diag}. The following are equivalent:
		\begin{enumerate}
			\item $\Delta$ is coassociative;
			\item $L_{3,k}=R_{3,k}$ for all $k\in Q_0$;
			\item $L_{n,k}=R_{n,k}$ for all $k\in Q_0$ and all $n\geq 1$.
		\end{enumerate}
	\end{corollary}
	
	\begin{proof}
		(1) $\Rightarrow$ (2) and (3) $\Rightarrow$ (1) are trivial. (2) $\Rightarrow$ (3) follows from Lemma~\ref{EquivalentDescriptionOfCAThroughPartition}.
	\end{proof}
	
	For the remainder of this section, we assume that $\Delta$ is a coassociative partitioning morphism with the property \eqref{eq:diag}. Such a $\Delta$ always exists; see \cite{H08A}. 
	In fact, let $\Delta:\bfk Q\to \bfk Q\otimes_\bfk \bfk Q$ be the partitioning morphism corresponding to the partition
	\[\{E_k=\{(k,i)\mid i\in Q_0\}\mid k\in Q_0\}.\]
	Then this $\Delta$ satisfies all the required assumptions.

	Denote the irreducible representation corresponding to the vertex $k\in Q_0$ by $M(\mathbf{1}_k)$.
	The following Lemma~gives the relation between two kinds of tensor product defined above and in Section~\ref{sec:tensor1}.
	\begin{lemma}[{\cite[Theorem~1]{H08A}}]\label{relationBTT_1&T_2}
		For $\bfk Q$-modules $M,N$, there is an isomorphism of $\bfk Q$-modules
		\[
		M\otimes^\Delta N\cong (M\otimes N)\oplus \bigoplus_{k\in Q_0}d_k M(\mathbf{1}_k)
		\]
		where 
		\[
		d_k=\sum\limits_{(k,k)\neq(i,j)\in E_k}\dim M_i\dim N_j,\quad \forall k\in Q_0.
		\]
	\end{lemma}
	
	\begin{example}\label{T_2WithSimpleModuleInGeneral}
		Let $M$ be a representation of $Q$. By a direct computation, we have
		\begin{align*}
			M(\mathbf{1}_k)\otimes^\Delta M=&(\dim M_k)M(\mathbf{1}_k)\oplus \bigoplus_{i\in Q_0}\big(\sum\limits_{(k,k)\neq(k,j)\in E_i}\dim M_j\big)M(\mathbf{1}_i)\\
			=&\bigoplus_{i\in Q_0}\big(\sum\limits_{(k,j)\in E_i}\dim M_j\big)M(\mathbf{1}_i)
		\end{align*} 
		and
		\begin{align*}
			M\otimes^\Delta M(\mathbf{1}_k)=&(\dim M_k)M(\mathbf{1}_k)\oplus \bigoplus_{i\in Q_0}\big(\sum\limits_{(k,k)\neq (j,k)\in E_i}\dim M_j\big)M(\mathbf{1}_i)
			\\
			=&\bigoplus_{i\in Q_0}\big(\sum\limits_{(j,k)\in E_i}\dim M_j\big)M(\mathbf{1}_i).
		\end{align*}
	\end{example}
	
	From this fact, we derive the relationship between the decompositions of these two types of tensor powers.
	\begin{proposition}\label{prop:DecompositionOfTwoKindsOfTensorProduct-n}
		Let $M$ be a representation of $Q$, $n$ be a positive integer. Then 
		\begin{align*}
			M^{\otimes^\Delta n}=M^{\otimes n}\oplus \bigoplus_{k\in Q_0}\Big(\sum_{(k,k,\cdots,k)\neq (a_1,a_2,\cdots,a_n)\in L_{n,k}} \dim M_{a_1} \dim M_{a_2}\cdots \dim M_{a_n}\Big)M(\mathbf{1}_k).
		\end{align*}
	\end{proposition}
	\begin{proof}
		The proof is established by induction on $n$.
		The case $n=1$ is trivial, and $n=2$ is exactly Lemma~\ref{relationBTT_1&T_2} for $M=N$.
		Assume now that the result holds for $n$. Let us show that it also holds for $n+1$.
		By the inductive assumption and the associativity of $\otimes^\Delta$, we have
		\begin{align*}
			&M^{\otimes^\Delta (n+1)}=M^{\otimes^\Delta n}\otimes^\Delta M\\
			=&\Big(M^{\otimes n}\oplus \bigoplus_{k\in Q_0}\big(\sum\limits_{(k,k,\cdots,k)\neq (a_1,a_2,\cdots,a_n)\in L_{n,k}}\dim M_{a_1}\dim M_{a_2}\cdots \dim M_{a_n}\big)M(\mathbf{1}_k)\Big)\otimes^\Delta M\\
			=&\big(M^{\otimes n}\otimes^\Delta M\big)\oplus\\&
			\bigoplus_{k\in Q_0}\Big(\big(\sum\limits_{(k,k,\cdots,k)\neq (a_1,a_2,\cdots,a_n)\in L_{n,k}}\dim M_{a_1}\dim M_{a_2}\cdots \dim M_{a_n}\big)M(\mathbf{1}_k)\otimes^\Delta M\Big)\\
			=&\Big(M^{\otimes (n+1)}\oplus \bigoplus_{k\in Q_0}\big(\sum\limits_{(k,k)\neq(a,b)\in L_{2,k}}(\dim M_a)^n\dim M_b\big)M(\mathbf{1}_k)\Big)\oplus\\
			&\bigoplus_{i\in Q_0}\Big(\sum\limits_{k\in Q_0}\sum\limits_{(k,k,\cdots,k)\neq (a_1,a_2,\cdots,a_n)\in L_{n,k}}\sum\limits_{(k,j)\in L_{2,i}}\dim M_{a_1}\dim M_{a_2}\cdots \dim M_{a_n}\dim M_j\Big)M(\mathbf{1}_i)\\
			=&\Big(M^{\otimes (n+1)}\oplus \bigoplus_{k\in Q_0}\big(\sum\limits_{(k,k)\neq(a,b)\in L_{2,k}}(\dim M_a)^n\dim M_b\big)M(\mathbf{1}_k)\Big)\oplus\\
			&\bigoplus_{k\in Q_0}\Big(\sum\limits_{i\in Q_0}\sum\limits_{(i,i,\cdots,i)\neq(a_1,\cdots,a_n)\in L_{n,i}}\sum\limits_{(i,j)\in L_{2,k}}\dim M_{a_1}\dim M_{a_2}\cdots \dim M_{a_n}\dim M_j\Big)M(\mathbf{1}_k)\\
			=&M^{\otimes (n+1)}\oplus \bigoplus_{k\in Q_0}\Big(\sum\limits_{(k,\cdots,k)\neq(a_1,\cdots,a_{n+1})\in L_{n+1,k}}\dim M_{a_1}\dim M_{a_2}\cdots \dim M_{a_{n+1}}\Big)M(\mathbf{1}_k),
		\end{align*}
		where we use Lemma~\ref{relationBTT_1&T_2} and Example \ref{T_2WithSimpleModuleInGeneral}.
		Hence, the desired formula follows by induction.
	\end{proof}
	
	Now we study the family of sets $\{L_{n,k}\mid k\in Q_0\}$.
	\begin{lemma}\label{TheClassOfLnkIsAPartionOfQ0n}
		For $n\geq2$, the family $\{L_{n,k}\mid k\in Q_0\}$ constitutes a partition of $Q_0^{\times n}$.
	\end{lemma}
	\begin{proof}
		We need to prove that $L_{n,k}$ are pairwise disjoint and \begin{align*}
			\bigsqcup_kL_{n,k}=\{(a_1,\cdots,a_n)\mid a_i\in Q_0,1\leq i\leq n\}.
		\end{align*}
		
		Suppose $(a_1,\cdots,a_n)\in L_{n,k_1}\cap L_{n,k_2}\neq \emptyset$. Then there exist $a_i',b_i',3\leq i\leq n$, such that
		\begin{align*}
			&(a_1,a_2)\in E_{a_3'}\cap E_{b_3'},\\
			&(a_3',a_{3})\in E_{a_4'},\cdots, (a_{n-1}',a_{n-1})\in E_{a_n'},\quad(a_n',a_n)\in E_{k_1},\\
			&(b_3',a_{3})\in E_{b_4'},\cdots, (b_{n-1}',a_{n-1})\in E_{b_n'},\quad (b_n',a_n)\in E_{k_2}.
		\end{align*}
		Since $\{E_k\mid k\in Q_0\}$ is a partition of $Q_0^{\times 2}$, for each $(a,b)$, there exists a unique $k\in Q_0$ such that $(a,b)\in E_k$.
		Thus, we have $a_3' = b_3'$, $a_4' = b_4'$, and so on. Proceeding recursively, we obtain $k_1 = k_2$. It then follows that $L_{n,k_1} = L_{n,k_2}$ whenever $L_{n,k_1} \cap L_{n,k_2}$ is non-empty.
		
		Given an arbitrary tuple $(a_1,\cdots,a_n)\in Q_0^{\times n}$, the partition $\{E_k\mid k\in Q_0\}$ of $Q_0^{\times 2}$ guarantees the existence of a sequence 
		$a_i'$ $(3\leq i\leq n+1)$ satisfying
		\[
		(a_1,a_2)\in E_{a_3'}\quad(a_i',a_{i})\in E_{a_{i+1}'},\quad\forall 3\leq i\leq n.
		\]
		Consequently, $(a_1,\cdots,a_n)\in L_{n,a_{n+1}'}$.
	\end{proof}
	
	\begin{corollary}\label{cor:RelationBetweenb&bDelta}
		Let $\Delta$ be a partitioning morphism with the property \eqref{eq:diag}, $M$ be a representation of $Q$, $n$ a positive integer. Then
		\begin{align*}
			b^\Delta_n(M)=b_n(M)+(\dim M)^n-\sum\limits_{k\in Q_0} (\dim M_k)^n.
		\end{align*}
		In particular, $b_n^\Delta(M)$ is independent of the choice of $\Delta$.
	\end{corollary}
	\begin{proof}
		By definition and Proposition~\ref{prop:DecompositionOfTwoKindsOfTensorProduct-n}, we have
		\begin{align}\label{RelationBTb_n&b_nD0}
			b_n^\Delta(M)=b_n(M)+\sum\limits_{k\in Q_0}\sum\limits_{(k,\cdots,k)\neq (a_1,\cdots,a_n)\in L_{n,k}}\dim M_{a_1}\dim M_{a_2}\cdots \dim M_{a_n}.
		\end{align}
		From \eqref{RelationBTb_n&b_nD0} and Lemma~\ref{TheClassOfLnkIsAPartionOfQ0n}, we obtain
		\begin{align*}
			b^\Delta_n(M)&=b_n(M)+\sum\limits_{(a_1,a_2,\cdots,a_n)\in Q_0^{\times n}}\dim M_{a_1}\dim M_{a_2}\cdots \dim M_{a_n}-\sum\limits_{k\in Q_0}(\dim M_k)^n\\
			&=b_n(M)+(\dim M)^n-\sum\limits_{k\in Q_0} (\dim M_k)^n.
		\end{align*}
		This completes the proof.
	\end{proof}
	
	\subsection{Proof of Theorem~\ref{thm:main1}}\label{sec:proof-1}
	Now, we can complete the proof of Theorem~\ref{thm:main1}.
	
	\begin{proof}[Proof of Theorem~\ref{thm:main1}]
		By Corollary~\ref{cor:RelationBetweenb&bDelta}, we have
		\begin{equation}\label{eq:SpecificExpressionOfBetaDelta}
			\beta^\Delta(M) = \lim_{n\to +\infty} \sqrt[n]{b_n^\Delta(M)} = \lim_{n\to +\infty} \sqrt[n]{b_n(M) + (\dim M)^n - \sum_{i\in Q_0} (\dim M_i)^n}.
		\end{equation}
		Since $\dim M = \sum_{i\in Q_0} \dim M_i$, Lemma~\ref{sqrtnlimOfSum}~(3) yields
		\begin{equation}\label{eq:TheLimitsOfWholeMinusParts}
			\lim_{n\to +\infty} \sqrt[n]{(\dim M)^n - \sum_{i\in Q_0} (\dim M_i)^n} = \dim M.
		\end{equation}
		Combining \eqref{eq:TheLimitsOfWholeMinusParts}, Theorem~\ref{thm:main0} and Lemma~\ref{sqrtnlimOfSum}~(1) with \eqref{eq:SpecificExpressionOfBetaDelta}, we obtain the desired limit.
	\end{proof}
	
	\section{Explicit formulas for Dynkin quivers of type \texorpdfstring{$A$}{A} and \texorpdfstring{$D$}{D}}
	
	\label{sec:example}

	In this section, we determine the exact formula of $b_n(M)$ (and also $b_n^\Delta(M)$)  
	for any finite-dimensional representation $M$ of a quiver of type $A$ or $D$.
	
	Let $l$ be a positive integer. 
	Unless stated otherwise, $Q$ is a quiver of finite type with $l$ vertices and $\Phi^+=\Phi^+(Q)$ is its set of positive roots.
	By Gabriel's theorem, the indecomposable representations of $Q$ are in bijection with the positive roots of $Q$ (up to isomorphism). For each positive root $\bfd \in \Phi^+$, we fix an indecomposable representation with dimension vector $\bfd$ and denote it by $M^Q(\bfd)$. When $Q$ is clear from the context, we simply write $M(\bfd)$. 
	Since the tensor product is well defined on isomorphism classes, this choice does not affect the decomposition results; hence we may safely work with these representatives.
	
	Let $M=\bigoplus\limits_{\bfd\in \Phi^+}a_\bfd M(\bfd)$ be a representation of $Q$. In general, its $n$-fold tensor product decomposes as
	\begin{equation}\label{eq:generalDecomposationOfDynkinType}
		M^{\otimes n}=\bigoplus_{\bfd\in \Phi^+}\left(
		\sum_{\substack{\bfd_1,\cdots,\bfd_n\in \Phi^+ ,\\\exists k,M' ,\text{ such that }M(\bfd_1)\otimes\cdots  \otimes M(\bfd_n)=kM(\bfd)\oplus M'\\\text{and }M(\bfd)\text{ is not a direct summand of }M'}}
		ka_{\bfd_1}\cdots a_{\bfd_n}
		\right)M(\bfd).
	\end{equation}
	We denote the coefficient of $M(\bfd)$ in \eqref{eq:generalDecomposationOfDynkinType} by $a_\bfd^{(n)}$, that is,
	\begin{align*}
		\label{eq:Mtensorn}
		M^{\otimes n}=\bigoplus\limits_{\bfd\in \Phi^+}a^{(n)}_{\bfd} M(\bfd).
	\end{align*}
	
	\subsection{Thin representations}\label{sec:ThinRoots}
	The discussion in this subsection generalizes \cite[Theorem~2]{H08B}.
	
	\begin{definition}
		A vector $\bfd \in \N^l$ $(l \ge 1)$ is called \emph{thin} if $\bfd \in \{0,1\}^{\times l}$.
	\end{definition}
	
	Let $\Phi^+_{\mathrm{th}}$ denote the subset of $\Phi^+$ consisting of the positive thin roots of $Q$.
	\begin{example}\label{ex:PositiveRoots-TypeA}
		For type $A$, $\Phi^+ = \Phi^+_{\mathrm{th}}$.
	\end{example}
	For each vertex $k$, write $\mathbf{1}_k$ for the corresponding simple root.
	
	\begin{definition}
		For a vector $\bfd = (d_i) \in \N^l$, its \emph{support} $\supp(\bfd)$ is the subset of $Q_0$ given by 
		\[\{i \in Q_0 \mid \ d_i \neq 0\},\]
		and its \emph{length} is 
		\[\ell(\bfd) = \sum_{i=1}^l d_i.\]
		We say $\bfd$ is \emph{of length $m$} if $\ell(\bfd) = m$. 
	\end{definition}
	
	Note that a positive root of $Q$ has length $1$ if and only if it is a simple root, and any positive root of length less than $5$ (the minimal length of a non‑thin positive root) is necessarily thin.
	
	We identify a subset $S \subseteq Q_0$ with the full subquiver of $Q$ it spans.
	
	Such a subset is called \emph{connected} if the induced full subquiver is connected; \emph{connected components} are defined analogously.
	Thin vectors in $\N^l$ correspond bijectively to subsets of $Q_0$, and hence to full subquivers of $Q$, via
	\begin{equation}\label{eq:ThinVector<->SubsetOfVertexSet}
		\bfd \mapsto \supp(\bfd), \qquad S \mapsto \mathbf{1}_S,
	\end{equation}
	where for a subset $S \subseteq Q_0$, the thin vector $\mathbf{1}_S = (d_i)$ is given by
	\[
	d_i = \begin{cases}
		1, & \text{if } i \in S,\\
		0, & \text{otherwise}.
	\end{cases}
	\]
	(The notation is justified because $\mathbf{1}_S$ is the dimension vector of the unit object for the tensor product in $\rep_\bfk(Q_S)$, where $Q_S$ denotes the full subquiver generated by $S$.) 
	In particular, the set $\Phi^+_{\mathrm{th}}$ of positive thin roots is in bijection with the connected subsets of $Q_0$, i.e., with the connected full subquivers of $Q$. Moreover, thin roots of length $i$ correspond bijectively to connected full subquivers on $i$ vertices.
	
	Recall that there is a partial order on $\Phi^+$: for $\bfd=(d_i),\bfd'=(d'_i) \in \Phi^+$, we write $\bfd \ge \bfd'$ if $d_i \ge d_i'$ for all $i$, and $\bfd > \bfd'$ if $\bfd \ge \bfd'$ and $\bfd \neq \bfd'$. In the case where $\bfd' \in \Phi^+_{\mathrm{th}}$, we have $\bfd \ge \bfd'$ if and only if $\supp(\bfd') \subseteq \supp(\bfd)$. 
	
	\begin{lemma}\label{lem:SubquiverCorrespondToDifferentLength}
		Let $\bfd\in \Phi^+$. Then we have the following bijections:
		\begin{enumerate}
			\item $\{\bfd'\in \Phi^+\mid \ell(\bfd')=1,\ \bfd'\le \bfd\} \stackrel{1:1}{\longleftrightarrow} (\supp(\bfd))_0$;
			\item $\{\bfd'\in \Phi^+\mid \ell(\bfd')=2,\ \bfd'\le \bfd\} \stackrel{1:1}{\longleftrightarrow}  (\supp(\bfd))_1$.
		\end{enumerate}
	\end{lemma}
	
	\begin{proof}
		A connected full subquiver on a single vertex is exactly that vertex, as there are no loops. 
		Since Dynkin diagrams are trees, there is at most one edge between any two vertices of $Q$; consequently, connected full subquivers on two vertices correspond bijectively to the arrows of $Q$. 
		The claimed bijections then follow directly from the preceding discussion.
	\end{proof}
	
	For two vectors $\bfd=(d_i)$ and $\bfd'=(d'_i)\in \N^l$, we define the pointwise product $\bfd*\bfd'$  by
	\[
	(\bfd*\bfd')_i = d_i d'_i.
	\]
	
	\begin{lemma}\label{lem:DimensionVectorOfTensorProduct}
		Let $Q$ be an arbitrary quiver, and let $M,N$ be two representations of $Q$.
		Then the vector $\underline{\dim}(M\otimes N)=\underline{\dim} M*\underline{\dim}N$.
	\end{lemma}
	\begin{proof}
		By definition of the diagonal tensor product of representations, for each vertex $i\in Q_0$ we have $\dim (M\otimes N)_i=\dim (M_i\otimes N_i)=\dim M_i \dim N_i$.
	\end{proof}
	\begin{corollary}\label{cor:TensorProductOfIndRepOfA}
		Let $\bfd,\bfd'$ be two thin roots of $Q$. Then the tensor product $M(\bfd)\otimes M(\bfd')=M(\bfd*\bfd')$.
	\end{corollary}
	\begin{proof}
		By definition and Lemma~\ref{lem:DimensionVectorOfTensorProduct}, the representation $M(\bfd)\otimes M(\bfd')$ is indecomposable and has dimension vector $\bfd * \bfd'$, as desired.
	\end{proof}

	For a positive root $\bfd\in \Phi^+$, consider the subset of summation indices in \eqref{eq:generalDecomposationOfDynkinType}
	\[
	S(\bfd)=\Bigl\{ \bfd' \in \Phi^+ \;\Big|\;
	\begin{aligned}[t]
		&\exists \bfd''\in \Phi^+,\ \exists k>0\ \text{such that } M(\bfd') \otimes M(\bfd'') = kM(\bfd)\oplus M',\\
		&\text{where }M(\bfd) \text{ is not a direct summand of } M'\Bigr\}.
	\end{aligned}
	\]
	
	For a nonzero vector $\bfd=(d_i)\in\N^l$, denote \[m(\bfd)=\min\{d_i\mid 1\leq i\leq l, d_i>0\}.\]
	The main tools we use to prove the formulas \eqref{eq:bForD} and \eqref{eq:bForA} are the following.
	
	\begin{proposition}\label{prop:SumOfCoefficientsOfGeneratorsOfThinRepOfTypeAandD}
		Let $Q$ be a quiver of type $A$ or $D$,  
		$M$ be a representation of $Q$ and 
		\[
		M=\bigoplus\limits_{\bfd\in \Phi^+}a_\bfd M(\bfd),\quad M^{\otimes n}=\bigoplus\limits_{\bfd\in \Phi^+}a_\bfd^{(n)}M(\bfd).
		\]
		Then for each $\bfd\in \Phi^+_{\mathrm{th}}$, 
		\begin{align}
			\Bigl(\sum_{\bfd'\in S(\bfd)} m(\bfd'*\bfd)\, a_{\bfd'}\Bigr)^{\!n} = \sum_{\bfd'\in S(\bfd)} m(\bfd'*\bfd)\, a_{\bfd'}^{(n)}.
		\end{align}
	\end{proposition}
	
	The proof of Proposition~\ref{prop:SumOfCoefficientsOfGeneratorsOfThinRepOfTypeAandD} is given explicitly for type $D$ in Appendix~\ref{sec:proofD}; the type $A$ case follows as a corollary. The proof for type $D$ relies on the tensor product decomposition formulas from \cite[Proposition 5--7]{H09}, which are used directly without repetition.
	
	We also need the following statement about quivers.
	\begin{lemma}\label{lem:NumberOfVerticesMinusNumberOfArrows}
		Let $Q$ be a finite connected quiver whose underlying graph is a tree. Then $\sharp Q_0 - \sharp Q_1 = 1$.
	\end{lemma}
	
	\begin{proof}
		We proceed by induction on $\sharp Q_0$. 
		The case $\sharp Q_0 = 1$ is trivial. 
		Assume the statement holds for all quivers with $n$ vertices ($n\ge1$), and let $Q$ be a quiver with $n+1$ vertices whose underlying graph is a tree. 
		Since the underlying graph is a finite tree, there exists a vertex $i_0 \in Q_0$ incident to exactly one edge. Let $Q'$ be the full subquiver of $Q$ with vertex set $Q_0 \setminus \{i_0\}$. 
		Then $Q'$ is connected: if $Q'$ had more than one connected component, the unique edge incident to $i_0$ could connect $i_0$ to at most one of them, leaving the other components disconnected from the rest of $Q$, contradicting the connectedness of $Q$. 
		Moreover, $Q'$ is acyclic because its underlying graph is a subtree of a tree. 
		By the induction hypothesis, $\sharp Q'_0 - \sharp Q'_1 = 1$. Since $Q$ is obtained from $Q'$ by adding one vertex and one edge, we have $\sharp Q_0 = \sharp Q'_0 + 1$ and $\sharp Q_1 = \sharp Q'_1 + 1$. Therefore $\sharp Q_0 - \sharp Q_1 = (\sharp Q'_0 + 1) - (\sharp Q'_1 + 1) = 1$.
	\end{proof}
	
	In particular, Lemma~\ref{lem:NumberOfVerticesMinusNumberOfArrows} holds for quivers of Dynkin types.
	
	\subsection{Embedding of positive roots}
	Let $Q'$ be a subquiver of $Q$. 
	Write $\sharp Q'_0 = l'$, then $l' \le l$.
	Define the embedding functor $H = H^Q_{Q'} : \rep_\bfk ( Q' )  \to \rep_\bfk ( Q ) $ by
	\[
	(H(M))_i = \begin{cases}
		M_i, & i \in Q'_0,\\
		0, & \text{otherwise},
	\end{cases}
	\qquad
	(H(M))_\alpha = \begin{cases}
		M_\alpha, & \alpha \in Q'_1,\\
		0, & \text{otherwise},
	\end{cases}
	\]
	for each vertex $i\in Q_0$ and arrow $\alpha\in Q_1$; and for a morphism $f = (f_i)$ in $\rep_\bfk ( Q' ) $,
	\[
	(H(f))_i = \begin{cases}
		f_i, & i \in Q'_0,\\
		0, & \text{otherwise}.
	\end{cases}
	\]
	
	The inclusion $Q'_0 \subseteq Q_0$ also induces an embedding of the corresponding lattices:
	\[
	H = H^Q_{Q'} : \N^{l'} \to \N^l,\qquad \bfd = (d_i) \mapsto H(\bfd),
	\]
	where $H(\bfd)$ is given by
	\[
	H(\bfd)_i = \begin{cases}
		d_i, & i \in Q'_0,\\
		0, & \text{otherwise}.
	\end{cases}
	\]
	
	\begin{proposition}\label{prop:PropertiesOfTheImbeddingFunctor}
		Let $Q$, $Q'$, $H$ be as above. Then
		\begin{enumerate}
			\item $H$ is fully faithful, additive, and monoidal;
			\item $\underline{\dim}(H(M)) = H(\underline{\dim}(M))$. This justifies the notation.
		\end{enumerate}
	\end{proposition}
	\begin{proof}
		Both statements follow directly from the definitions.
	\end{proof}
	
	\begin{corollary}\label{cor:EmbeddingOfPositiveRoots}
		Suppose $Q'$ is a Dynkin subquiver of the Dynkin quiver $Q$. Let $\bfd$ be a positive root of $Q'$. Then
		\begin{enumerate}
			\item $H(\bfd)$ is a positive root of $Q$.
			\item $H(M^{Q'}(\bfd)) \cong M^{Q}(H(\bfd))$.
		\end{enumerate}
	\end{corollary}
	\begin{proof}
		Since $H$ is fully faithful, it induces an isomorphism of algebras
		\[
		\End_{Q}\bigl(H(M^{Q'}(\bfd))\bigr) \cong \End_{Q'}\bigl(M^{Q'}(\bfd)\bigr).
		\]
		Because $M^{Q'}(\bfd)$ is indecomposable, its endomorphism algebra is local; so is that of $H(M^{Q'}(\bfd))$, hence $H(M^{Q'}(\bfd))$ is indecomposable. 
		By Proposition~\ref{prop:PropertiesOfTheImbeddingFunctor}, the dimension vector of $H(M^{Q'}(\bfd))$ equals $H(\bfd)$. 
		Now Gabriel's theorem implies that $H(\bfd)$ is a positive root of $Q$, which proves (1). Moreover, for a Dynkin quiver, indecomposable representations are uniquely determined (up to isomorphism) by their dimension vectors. Consequently, $H(M^{Q'}(\bfd)) \cong M^{Q}(H(\bfd))$, establishing (2).
	\end{proof}
	
	\begin{corollary}\label{cor:EmbeddingOfDecomposition}
		Suppose $Q'$ is a Dynkin subquiver of the Dynkin quiver $Q$.
		\begin{enumerate}
			\item Let $M = \bigoplus_{\bfd\in \Phi^+(Q')} a_{\bfd} M^{Q'}(\bfd)$ be a representation of $Q'$. Then 
			\[H(M) \cong \bigoplus_{\bfd\in \Phi^+(Q')} a_{\bfd} M^{Q}(H(\bfd)).\]
			\item Let $M,N$ be representations of $Q'$ such that $M \otimes N \cong \bigoplus_{\bfd\in \Phi^+(Q')} a_{\bfd} M^{Q'}(\bfd)$. Then 
			\[H(M) \otimes H(N) \cong \bigoplus_{\bfd\in \Phi^+(Q')} a_{\bfd} M^{Q}(H(\bfd)).\]
		\end{enumerate}
	\end{corollary}
	\begin{proof}
		(1) follows directly from Proposition~\ref{prop:PropertiesOfTheImbeddingFunctor}~(1) and Corollary~\ref{cor:EmbeddingOfPositiveRoots}~(2). 
		(2) follows from (1) together with Proposition~\ref{prop:PropertiesOfTheImbeddingFunctor}~(1).
	\end{proof}

	\subsection{Dynkin quivers of type \texorpdfstring{$D$}{D}}
	In this subsection, we assume that the quiver $Q$ is of type $D_l$ $(l\geq 4)$.
	We assume that the underlying graph of $Q$ is
	
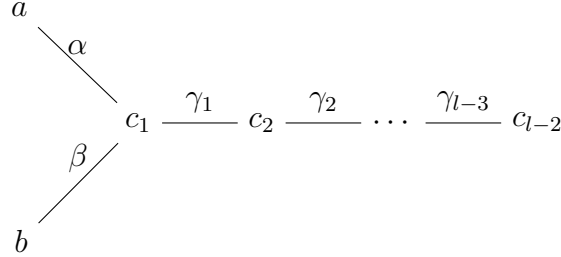
\begin{center}
		\tikz{
			\node (1) at (0,0) {$c_1$};
			\node[above left =of 1] (a) {$a$};
			\node[below left =of 1] (b) {$b$};
			\node[right=of 1] (2) {$c_2$};
			\node[right=of 2] (3) {$\cdots$};
			\node[right=of 3] (4) {$c_{l-2}$};
			
			\path 
			(1) edge node[above] {$\alpha$}  (a) edge node[above] {$\beta$}  (b)
			(3) edge node[above] {$\gamma_{l-3}$} (4);
			\foreach \x [remember=\x as \lastx (initially 1)] in {2,3} {\path (\lastx) edge node[above] {$\gamma_\lastx$} (\x);};
		}
	\end{center}
	
	The set $\Phi^+$ of positive roots of $Q$ is partitioned into thin roots $\Phi^+_{\mathrm{th}}$ and twin roots $\Phi^+_{\mathrm{tw}}$. In \cite[Section~2.4]{R16}, twin roots are further classified as \emph{$r$-twin roots} ($1\le r\le l-3$) of the form
	\[
	\mathbf{1}_a+\mathbf{1}_b+\sum_{k=1}^{i}\mathbf{1}_k+\sum_{k=1}^{r}\mathbf{1}_k,\qquad 1\leq r<i\le l-2.
	\]
	For brevity we simply call them \emph{twin roots}. Their general form is
	\begin{equation}\label{eq:DefinitionOfTwinRoots}
		\bfx_{i,j}=\mathbf{1}_a+\mathbf{1}_b+\sum_{k=1}^{i}\mathbf{1}_k+\sum_{k=1}^{j}\mathbf{1}_k,
	\end{equation}
	where $i,j\in\N$ and $1\le i<j\le l-2$.
	The same symbol $\bfx_{i,j}$ is used for quivers of different types $D_l$; its dimension vector acquires trailing zeros as $l$ grows. This apparent ambiguity does not affect our computations of the tensor product, thanks to Corollaries~\ref{cor:EmbeddingOfPositiveRoots} and \ref{cor:EmbeddingOfDecomposition}. Indeed, If we denote by $\bfx^{Q}_{i,j}$ the twin root in $Q$, then for the full subquiver $Q'$ of type $D_{l'}$ $(l \ge l')$ we have 
	\[\bfx^{Q}_{i,j} = H(\bfx^{Q'}_{i,j}).\]
	Since the embedding functor $H$ preserves the decompositions of the tensor product (see Corollary~\ref{cor:EmbeddingOfDecomposition}), the simplified notation is justified.
	
	Note that the bijection \eqref{eq:ThinVector<->SubsetOfVertexSet} does not apply to twin roots. In fact, the vector 
	\[
	\mathbf{1}_{\supp(\bfx_{i,j})}=\mathbf{1}_a+\mathbf{1}_b+\sum\limits_{k=1}^{j}\mathbf{1}_k
	\]
	is a thin root, whereas $\bfx_{i,j}$ is not.
	
	For $\bfd', \bfd \in \Phi^+$, by definition, $m(\bfd' * \bfd) = 2$ only when one of $\bfd',\bfd$ is a twin root $\bfx_{i,j}$ and the other is a thin root whose support is contained in $\{k\in Q_0\mid 1\le k\le i\}$; otherwise the value is $1$.
	We remark that 
	\begin{equation}\label{eq:SumOfM(-*1)a}
		\sum\limits_{\bfd\geq \mathbf{1}_k}m(\bfd*\mathbf{1}_k)a_\bfd=\dim M_k
	\end{equation}
	for all vertex $k\in Q_0$.
	
	\begin{theorem}\label{thm:bForTypeD}
		Let $n$ be a positive integer, let $M$ be a representation of $Q$, and write
		\[
		M = \bigoplus_{\bfd\in\Phi^+} a_\bfd M(\bfd), \qquad 
		M^{\otimes n} = \bigoplus_{\bfd\in\Phi^+} a_\bfd^{(n)} M(\bfd).
		\]
		Then
		\begin{equation}\label{eq:bForD}
			b_n(M) = \sum_{i\in Q_0} (\dim M_i)^n
			- \sum_{\substack{\bfd\in\Phi^+ \\ \ell(\bfd)=2}} \Bigl( \sum_{\substack{\bfd'\in\Phi^+ \\ \bfd'\ge\bfd}} m(\bfd'*\bfd)\, a_{\bfd'} \Bigr)^{\!n}
			- \sum_{\bfd\in\Phi^+_{\mathrm{tw}}} a_\bfd^{(n)}.
		\end{equation}
	\end{theorem}
	
	\begin{proof}
		Every positive root of length at most $2$ is thin. By Corollary~\ref{cor:TensorProductOfIndRepOfA} and \cite[Propositions 5--7]{H09}, for such $\bfd$ we have $S(\bfd)=\{\bfd'\in\Phi^+\mid \bfd'\ge\bfd\}$.
		
		We first apply Proposition~\ref{prop:SumOfCoefficientsOfGeneratorsOfThinRepOfTypeAandD} to the simple roots. From \eqref{eq:SumOfM(-*1)a}, this gives
		\begin{align}
			b_n(M) &= \sum_{i\in Q_0} \Bigl( (\dim M_i)^n - \sum_{\substack{\bfd\in\Phi^+ \\ \bfd > \mathbf{1}_i}} m(\bfd*\mathbf{1}_i)\, a_{\bfd}^{(n)} \Bigr) 
			+ \sum_{\substack{\bfd\in\Phi^+ \\ \ell(\bfd)\ge 2}} a_{\bfd}^{(n)} \nonumber\\
			&= \sum_{i\in Q_0} (\dim M_i)^n 
			- \sum_{\substack{\bfd\in\Phi^+ \\ \ell(\bfd)\ge 2}} \Bigl( \sum_{\substack{i\in Q_0 \\ \mathbf{1}_i \le \bfd}} m(\bfd*\mathbf{1}_i) - 1 \Bigr) a_{\bfd}^{(n)}. \label{eq:TypeDFormula-step1}
		\end{align}
		For a positive root $\bfd = (d_i)$, from Lemma~\ref{lem:SubquiverCorrespondToDifferentLength}~(1) we have
		\[
		\sum_{\substack{i\in Q_0 \\ \mathbf{1}_i\le \bfd}} m(\bfd*\mathbf{1}_i)
		= \sum_{i\in \supp(\bfd)} m(\bfd*\mathbf{1}_i)
		= \sum_{i\in \supp(\bfd)} d_i = \ell(\bfd).
		\]
		Inserting this into \eqref{eq:TypeDFormula-step1} we obtain
		\begin{equation}
			b_n(M) = \sum_{i\in Q_0} (\dim M_i)^n - \sum_{\substack{\bfd\in\Phi^+ \\ \ell(\bfd)\ge 2}} (\ell(\bfd)-1)\, a_{\bfd}^{(n)}. \label{eq:step2}
		\end{equation}
		
		Now we isolate the contribution of length‑2 roots. For any $\bfd$ with $\ell(\bfd)=2$, Proposition~\ref{prop:SumOfCoefficientsOfGeneratorsOfThinRepOfTypeAandD} gives
		\[
		a_{\bfd}^{(n)} = \Bigl( \sum_{\bfd'\ge\bfd} m(\bfd'*\bfd) a_{\bfd'} \Bigr)^{\!n} - \sum_{\bfd'>\bfd} m(\bfd'*\bfd) a_{\bfd'}^{(n)}.
		\]
		Summing over all such $\bfd$ we get
		\begin{align}
			\sum_{\ell(\bfd)=2} a_{\bfd}^{(n)}
			&= \sum_{\ell(\bfd)=2} \Bigl( \sum_{\bfd'\ge\bfd} m(\bfd'*\bfd) a_{\bfd'} \Bigr)^{\!n}
			- \sum_{\bfd'\in\Phi^+} \Bigl( \sum_{\substack{\ell(\bfd)=2 \\ \bfd<\bfd'}} m(\bfd'*\bfd) \Bigr) a_{\bfd'}^{(n)}. \label{eq:step3}
		\end{align}
		The inner coefficient depends on the type of $\bfd'$. Using Lemma~\ref{lem:SubquiverCorrespondToDifferentLength}(2) and Lemma~\ref{lem:NumberOfVerticesMinusNumberOfArrows}, one obtains
		\[
		\sum_{\substack{\ell(\bfd)=2 \\ \bfd<\bfd'}} m(\bfd'*\bfd) =
		\begin{cases}
			\ell(\bfd')-1, & \bfd'\in\Phi^+_{\mathrm{th}},\\[4pt]
			\ell(\bfd')-2, & \bfd'\in\Phi^+_{\mathrm{tw}}.
		\end{cases}
		\]
		Substituting these values into \eqref{eq:step3} yields
		\begin{equation}
			\sum_{\ell(\bfd)=2} a_{\bfd}^{(n)}
			= \sum_{\ell(\bfd)=2} \Bigl( \sum_{\bfd'\ge\bfd} m(\bfd'*\bfd) a_{\bfd'} \Bigr)^{\!n}
			- \sum_{\substack{\bfd\in\Phi^+_{\mathrm{th}} \\ \ell(\bfd)>2}} (\ell(\bfd)-1) a_{\bfd}^{(n)}
			- \sum_{\bfd\in\Phi^+_{\mathrm{tw}}} (\ell(\bfd)-2) a_{\bfd}^{(n)}. \label{eq:step4}
		\end{equation}
		
		Finally, combining \eqref{eq:step2} and \eqref{eq:step4}, the terms involving thin roots of length $>2$ cancel and we obtain
		\[
		b_n(M) = \sum_{i\in Q_0} (\dim M_i)^n
		- \sum_{\ell(\bfd)=2} \Bigl( \sum_{\bfd'\ge\bfd} m(\bfd'*\bfd) a_{\bfd'} \Bigr)^{\!n}
		- \sum_{\bfd\in\Phi^+_{\mathrm{tw}}} a_{\bfd}^{(n)},
		\]
		which is exactly \eqref{eq:bForD}.  This completes the proof.
	\end{proof}
	
	By \eqref{eq:bForD}, to derive the explicit formulas $b_n(M)$ for $Q$ of type $D$, it suffices to determine the summation $\sum\limits_{\bfd\in \Phi^+_{\mathrm{tw}}}a_{\bfd}^{(n)}$, which is  precisely characterized in Proposition \ref{cor:SumOfn-FoldTwinCoefficients}.
	
	\subsection{Dynkin quivers of type \texorpdfstring{$A$}{A}}
	\label{sec:bForTypeA}
	In this subsection, let $Q$ be a quiver of type $A_l$. Then $m(\bfd'*\bfd) = 1$ for all $\bfd,\bfd'\in \Phi^+$.
	
	The following Theorem~gives the formula of $b_n(M)$ for any representation $M$ of $Q$.
	\begin{theorem}\label{thm:bForTypeA}
		Let $n$ be a positive integer, $M$ be a representation of $Q$ and 
		\[
		M=\bigoplus\limits_{\bfd\in \Phi^+}a_\bfd M(\bfd),\quad M^{\otimes n}=\bigoplus\limits_{\bfd\in \Phi^+}a_\bfd^{(n)}M(\bfd).
		\]
		Then we have 
		\begin{align}\label{eq:bForA}
			b_n(M)=\sum\limits_{i=1}^l(\dim M_i)^n-\sum\limits_{\bfd\in \Phi^+,\,\ell(\bfd)=2}\Big(\sum\limits_{\bfd'\in \Phi^+,\,\bfd'\geq \bfd}a_{\bfd'}\Big)^n.
		\end{align}
	\end{theorem}
	
	\begin{proof}
		We regard $Q$ as a full subquiver of a quiver of type $D$ (by adding a suitable vertex and an edge). 
		Extend the representation $M$ of $Q$ to a representation of the larger $D$-type quiver by assigning the coefficient $0$ to all positive roots that are not roots of $Q$. 
		Then the result follows by applying Corollary~\ref{cor:EmbeddingOfDecomposition} and Theorem~\ref{thm:bForTypeD}.
	\end{proof}
	
	\begin{example}
		Take a representation $M=\oplus_{\bfd}a_\bfd M(\bfd)$ of $Q$ of type $A_3$. 
		Then the $n$-th term of the growth sequence of $M$ is given by
		\[
		b_n(M)=(\dim M_1)^n+(\dim M_2)^n+(\dim M_3)^n-(a_{110}+a_{111})^n-(a_{011}+a_{111})^n.
		\]
	\end{example}

	\appendix
	
	\section{Tensor product for type \texorpdfstring{$D$}{D}}

	In the following, we always assume that $Q$ is a quiver of type $D_l$ $(l\ge 4)$. 
	
	\subsection{Indecomposable representations of type $D$}
	
	Following \cite{H09}, define $\sigma: Q_1 \to \{0,1\}$ by $\sigma(\delta)=1$ if $\delta$ points to the branching point $c_1$, and $\sigma(\delta)=0$ otherwise.
	For all integers $i,j$ such that $1\le i\le j\le l-2$, let $A^i,B^i,C^{ij},D^i$ be the subquivers of $Q$ corresponding to the graphs
	\begin{center}
		\tikz[baseline=(label.base)]{
			\node (label) at (0,0) {$A^i\!:\qquad\qquad$}; 
			\node[right=of label] {\tikz{
					\node (c1) at (0,0) {$c_1$};
					\node[above left=of c1] (a) {$a$};
					\node[right=of c1] (dots) {$\cdots$};
					\node[right=of dots] (ci) {$c_{i},$};
					\path (c1) edge node[above] {$\alpha$} (a)
					(c1) edge node[above] {$\gamma_1$} (dots)
					(dots) edge node[above] {$\gamma_{i-1}$} (ci);
			}};}
		\\[6pt]
		
		\tikz[baseline=(label.base)]{
			\node (label) at (0,0) {$B^i\!:\qquad\qquad$};
			\node[right=of label] {\tikz{\node (c1) at (0,0) {$c_1$};
					\node[below left=of c1] (b) {$b$};
					\node[right=of c1] (dots) {$\cdots$};
					\node[right=of dots] (ci) {$c_{i},$};
					\path (c1) edge node[above] {$\beta$} (b)
					(c1) edge node[above] {$\gamma_1$} (dots)
					(dots) edge node[above] {$\gamma_{i-1}$} (ci);
			}};}
		\\[6pt]
		
		\tikz[baseline=(label.base)]{
			\node (label) at (0,0) {$C^{ij}\!:\qquad\qquad\qquad\qquad$}; 
			\node[right=of label] {\tikz{\node (ci) at (0,0) {$c_i$};
					\node[right=of ci] (dots) {$\cdots$};
					\node[right=of dots] (cj) {$c_j,$};
					\path (ci) edge node[above] {$\gamma_i$} (dots)
					(dots) edge node[above] {$\gamma_{j-1}$} (cj);
			}};}
		\\[6pt]
		
		\tikz[baseline=(label.base)]{
			\node (label) at (0,0) {$D^i\!:\qquad\qquad$};
			\node[right=of label] {\tikz{\node (c1) at (0,0) {$c_1$};
					\node[above left=of c1] (a) {$a$};
					\node[below left=of c1] (b) {$b$};
					\node[right=of c1] (dots) {$\cdots$};
					\node[right=of dots] (ci) {$c_i.$};
					\path (c1) edge node[above] {$\alpha$} (a)
					(c1) edge node[above] {$\beta$} (b)
					(c1) edge node[above] {$\gamma_1$} (dots)
					(dots) edge node[above] {$\gamma_{i-1}$} (ci);
			}};}
	\end{center}
	Then $\{\mathbf{1}_{A^i}\mid 0\leq i\leq l-2\}$, $\{\mathbf{1}_{B^i}\mid 0\leq i\leq l-2\}$, $\{\mathbf{1}_{C^{ij}}\mid 1\leq i<j\leq l-2\}$, and $\{\mathbf{1}_{D^i}\mid 1\leq i\leq l-2\}$ give all the thin roots of $Q$.
	
	For a twin root $\bfx_{i,j}$ of $Q$, we choose the corresponding indecomposable representation $M(\bfx_{i,j})$ (which is denoted by $X_{ij}$ in \cite{H09}) as
	\begin{center}
		\tikz{
			\node (3) at (0,0) {$\bfk^2$};
			\node[above left =of 3] (1) {$\bfk$};
			\node[below left =of 3] (2) {$\bfk$};
			\node[right=of 3] (4) {$\bfk^2$};
			\node[right=of 4] (5) {$\bfk$};
			\node[right=of 5] (6) {$\bfk$};
			\node[right=of 6] (7) {0};
			\node[right=of 7] (8) {0};
			
			\path 
			(3) edge node[above] {$A$} (1) 
			edge node[above] {$B$} (2) 
			(4) edge node[above] {$C$} (5)
			(6) edge (7);
			\path[dashed]
			(3) edge (4)
			(5) edge (6)
			(7) edge (8);
		}
	\end{center}
	where the dashed lines denotes several copies of the same vector spaces with the identity map between them. The linear maps $C$ is given by either of the matrices $\begin{pmatrix}
		1&1
	\end{pmatrix},\begin{pmatrix}
		1\\1
	\end{pmatrix}$ depending on the orientation of the corresponding arrow. Similarly the linear map $A$ is given by either of the matrices $\begin{pmatrix}
		0&1
	\end{pmatrix},\begin{pmatrix}
		1\\0
	\end{pmatrix}$ and $B$ by $\begin{pmatrix}
		1&0
	\end{pmatrix},\begin{pmatrix}
		0\\1
	\end{pmatrix}$ depending on the orientation of the corresponding arrows.
	
	The decomposition formulas for the tensor product of indecomposable representations are given in \cite[Propositions 5--7]{H09}.

	\subsection{The summation index set \texorpdfstring{$S(\bfd)$}{S(d)} for positive roots in type $D$}
	
	For the pointwise tensor product $\otimes$ and a positive root $\bfd$ of $Q$, let $S(n,\bfd)$ denote the set of $n$-tuples $(\bfd_1,\dots,\bfd_n)\in (\Phi^+)^{\times n}$ such that $M(\bfd)$ is a direct summand of $M(\bfd_1)\otimes\cdots\otimes M(\bfd_n)$.
	
	\begin{lemma}\label{lem:theLittleGenerateTheLarge}
		Let $\bfd \in \Phi^+$. Then 
		\[
		\bigcup_{\bfd'\in S(\bfd)}S(n,\bfd')\subset (S(\bfd))^{\times n}.
		\]
	\end{lemma}
	\begin{proof}
		Let $\bfd'\in S(\bfd)$ and take any $(\bfd_1,\dots,\bfd_n)\in S(n,\bfd')$, so that $M(\bfd')$ is a direct summand of $M(\bfd_1)\otimes\cdots\otimes M(\bfd_n)$. We need to show that $\bfd_i\in S(\bfd)$ for all $i$.
		
		Since $\bfd'\in S(\bfd)$, there exists $\bfd''\in \Phi^+$ such that $M(\bfd)$ is a direct summand of $M(\bfd')\otimes M(\bfd'')$. Because the pointwise tensor product is associative and distributes over direct sums, $M(\bfd)$ is a direct summand of
		\[
		\bigl( M(\bfd_1)\otimes\cdots\otimes M(\bfd_n) \bigr) \otimes M(\bfd'').
		\]
		Since $\otimes$ is commutative, and by Krull-Smithdt theorem, it follows that each $\bfd_i\in S(\bfd)$, hence $(\bfd_1,\dots,\bfd_n)\in (S(\bfd))^{\times n}$, proving the inclusion.
	\end{proof}

	\subsection{Proof of Proposition \ref{prop:SumOfCoefficientsOfGeneratorsOfThinRepOfTypeAandD} for type \texorpdfstring{$D$}{D}}
	\label{sec:proofD}

	\begin{proof}[Proof of Proposition~\ref{prop:SumOfCoefficientsOfGeneratorsOfThinRepOfTypeAandD} for type $D$]
		The proof proceeds by induction on $n$. The case $n=1$ is trivial.
		Assume that the statement holds for some $n \ge 1$.
		
		Fix $\bfd\in \Phi^+_{\mathrm{th}}$. 
		By the inductive hypothesis,
		\begin{align}
			\Bigl(\sum_{\bfd'\in S(\bfd)} m(\bfd'*\bfd)\,a_{\bfd'}\Bigr)^{\!n+1}
			&= \Bigl(\sum_{\bfd'\in S(\bfd)} m(\bfd'*\bfd)\,a_{\bfd'}\Bigr)^{\!n}
			\Bigl(\sum_{\bfd'\in S(\bfd)} m(\bfd'*\bfd)\,a_{\bfd'}\Bigr) \nonumber\\
			&= \Bigl(\sum_{\bfd'\in S(\bfd)} m(\bfd'*\bfd)\,a_{\bfd'}^{(n)}\Bigr)
			\Bigl(\sum_{\bfd'\in S(\bfd)} m(\bfd'*\bfd)\,a_{\bfd'}\Bigr) \nonumber\\
			&= \sum_{\bfd',\bfd''\in S(\bfd)} m(\bfd'*\bfd)\,m(\bfd''*\bfd)\,
			a_{\bfd'}^{(n)} a_{\bfd''}. \label{eq:step1}
		\end{align}
		
		For any $\bfd',\bfd''\in S(\bfd)$, write the tensor product decomposition
		\[
		M(\bfd')\otimes M(\bfd'') = \bigoplus_{\bfd'''\in\Phi^+} k_{\bfd'''}(\bfd',\bfd'')\, M(\bfd''').
		\]
		By Corollary~\ref{cor:TensorProductOfIndRepOfA} and \cite[Propositions 5--7]{H09}, the multiplicities satisfy
		\begin{equation}\label{eq:mult}
			m(\bfd'*\bfd)\,m(\bfd''*\bfd) = \sum_{\bfd'''\in S(\bfd)} k_{\bfd'''}(\bfd',\bfd'')\, m(\bfd'''*\bfd),
		\end{equation}
		which follows from a direct case‑by‑case verification. 
		
		Substituting \eqref{eq:mult} into \eqref{eq:step1} and interchanging the order of summation gives
		\[
		\sum_{\bfd',\bfd''\in S(\bfd)} m(\bfd'*\bfd)\,m(\bfd''*\bfd)\, a_{\bfd'}^{(n)} a_{\bfd''}
		= \sum_{\bfd'''\in S(\bfd)} m(\bfd'''*\bfd)\,
		\Bigl( \sum_{\bfd',\bfd''\in S(\bfd)} k_{\bfd'''}(\bfd',\bfd'')\, a_{\bfd'}^{(n)} a_{\bfd''} \Bigr).
		\]
		By Lemma~\ref{lem:theLittleGenerateTheLarge}, for each $\bfd'''\in S(\bfd)$, we have
		\[
		a_{\bfd'''}^{(n+1)} = \sum_{\bfd',\bfd''\in S(\bfd)} k_{\bfd'''}(\bfd',\bfd'')\, a_{\bfd'}^{(n)} a_{\bfd''}.
		\]
		
		Consequently,
		\[
		\Bigl(\sum_{\bfd'\in S(\bfd)} m(\bfd'*\bfd)\,a_{\bfd'}\Bigr)^{\!n+1}
		= \sum_{\bfd'\in S(\bfd)} m(\bfd'*\bfd)\, a_{\bfd'}^{(n+1)}.
		\]
		This completes the induction step, and the proposition follows for all $n$.
	\end{proof}
	
	\subsection{Sum of coefficients of twin representations}\label{sec:SumOfTwinRoots}
	In what follows, let $n$ be a positive integer, $M$ be a representation of $Q$, and 
	\[M=\bigoplus\limits_{\bfd\in \Phi^+}a_\bfd M(\bfd),\quad M^{\otimes n}=\bigoplus\limits_{\bfd\in \Phi^+}a_\bfd^{(n)}M(\bfd).\]
	By Theorem~\ref{thm:bForTypeD}, to obtain an explicit formula for $b_n(M)$, it suffices to determine the sum of the coefficients of twin roots
	\[\sum\limits_{1\le i<j\le l-2}a_{\bfx_{i,j}}^{(n)}\]
	in the $n$-fold tensor product of $M$.

	\begin{lemma}\label{lem:SumOfn-FoldTwinCoefficients-FixI-SameOrientation}
		If $\sigma(\alpha)=\sigma(\beta)$, let $P=\{1\leq k\leq l-3\mid \sigma(\gamma_{k})\neq \sigma(\alpha)\}$.
		Then for each (fixed) $1\leq i\leq l-3$, we have
		\begin{equation}\label{eq:SumOfn-FoldTwinCoefficients-FixI-SameOrientation}
			\sum_{j>i} a_{\bfx_{i,j}}^{(n)} = \begin{cases}
				n \biggl( \sum\limits_{j>i} a_{\bfx_{i,j}} \biggr)
				\biggl( \sum\limits_{\substack{p\in P,\; p< i \\ j>i}} a_{\bfx_{p,j}} + \sum\limits_{\substack{\bfd\in \Phi^+_{\mathrm{th}} \\ \bfd>\mathbf{1}_{D^i}}} a_\bfd \biggr)^{\!n-1}, & \text{ if } i\notin P,\\[6pt]
				\biggl( \sum\limits_{\substack{p\in P,\; p\le i \\ j>i}} a_{\bfx_{p,j}} + \sum\limits_{\substack{\bfd\in \Phi^+_{\mathrm{th}} \\ \bfd>\mathbf{1}_{D^i}}} a_\bfd \biggr)^{\!n}
				- \biggl( \sum\limits_{\substack{p\in P,\; p< i \\ j>i}} a_{\bfx_{p,j}} + \sum\limits_{\substack{\bfd\in \Phi^+_{\mathrm{th}} \\ \bfd>\mathbf{1}_{D^i}}} a_\bfd \biggr)^{\!n}, & \text{ if }i \in P.
			\end{cases}
		\end{equation}
	\end{lemma}
	
	\begin{proof}
		The proof proceeds by induction on $n$. The case $n=1$ is trivial.
		Assume that Equation \eqref{eq:SumOfn-FoldTwinCoefficients-FixI-SameOrientation} holds for a given $n$. We prove it for $n+1$.
		
		We claim that for each (fixed) \(1 \le i_0 \le l-3\), the equation
		\begin{equation}\label{eq:SumOfn-FoldTwinCoefficients-FixI-P0-SameOrientation}
			\sum_{\substack{p\in P,\; p<i_0 \\ j>i_0}} a_{\bfx_{p,j}}^{(n)} + \sum_{\substack{\bfd\in \Phi^+_{\mathrm{th}} \\ \bfd > \mathbf{1}_{D^{i_0}}}} a_\bfd^{(n)}
			= \biggl( \sum_{\substack{p\in P,\; p<i_0 \\ j>i_0}} a_{\bfx_{p,j}} + \sum_{\substack{\bfd\in \Phi^+_{\mathrm{th}} \\ \bfd > \mathbf{1}_{D^{i_0}}}} a_\bfd \biggr)^{\!n}
		\end{equation}
		holds.
		Then, by \cite[Propositions 5--7]{H09} and the inductive hypothesis, we obtain the following.
		
		For \(i \notin P\),
		\begin{align*}
			\sum_{j>i} a_{\bfx_{i,j}}^{(n+1)}
			&= \biggl( \sum_{j>i} a_{\bfx_{i,j}}^{(n)} \biggr)
			\biggl( \sum_{\substack{p\in P,\; p<i \\ j>i}} a_{\bfx_{p,j}} + \sum_{\substack{\bfd\in \Phi^+_{\mathrm{th}} \\ \bfd>\mathbf{1}_{D^i}}} a_\bfd \biggr)\\&\qquad  + \biggl( \sum_{\substack{p\in P,\; p<i \\ j>i}} a_{\bfx_{p,j}}^{(n)} + \sum_{\substack{\bfd\in \Phi^+_{\mathrm{th}} \\ \bfd>\mathbf{1}_{D^i}}} a_\bfd^{(n)} \biggr)
			\biggl( \sum_{j>i} a_{\bfx_{i,j}} \biggr) \\
			&= (n+1) \biggl( \sum_{j>i} a_{\bfx_{i,j}} \biggr)
			\biggl( \sum_{\substack{p\in P,\; p< i \\ j>i}} a_{\bfx_{p,j}} + \sum_{\substack{\bfd\in \Phi^+_{\mathrm{th}} \\ \bfd>\mathbf{1}_{D^i}}} a_\bfd \biggr)^{\!n}.
		\end{align*}
		
		For \(i \in P\),
		\begin{align*}
			\sum_{j>i} a_{\bfx_{i,j}}^{(n+1)}
			&= \biggl( \sum_{j>i} a_{\bfx_{i,j}}^{(n)} \biggr)
			\biggl( \sum_{\substack{p\in P,\; p\le i \\ j>i}} a_{\bfx_{p,j}} + \sum_{\substack{\bfd\in \Phi^+_{\mathrm{th}} \\ \bfd>\mathbf{1}_{D^i}}} a_\bfd \biggr) \\&\qquad + \biggl( \sum_{\substack{p\in P,\; p< i \\ j>i}} a_{\bfx_{p,j}}^{(n)}+ \sum_{\substack{\bfd\in \Phi^+_{\mathrm{th}} \\ \bfd>\mathbf{1}_{D^i}}} a_\bfd^{(n)} \biggr)
			\biggl( \sum_{j>p} a_{\bfx_{p,j}} \biggr) \\
			&= \biggl( \sum_{\substack{p\in P,\; p\le i \\ j>i}} a_{\bfx_{p,j}} + \sum_{\substack{\bfd\in \Phi^+_{\mathrm{th}} \\ \bfd>\mathbf{1}_{D^i}}} a_\bfd \biggr)^{\!n+1} - \biggl( \sum_{\substack{p\in P,\; p< i \\ j>i}} a_{\bfx_{p,j}} + \sum_{\substack{\bfd\in \Phi^+_{\mathrm{th}} \\ \bfd>\mathbf{1}_{D^i}}} a_\bfd \biggr)^{\!n+1}.
		\end{align*}
		
		Hence \eqref{eq:SumOfn-FoldTwinCoefficients-FixI-SameOrientation} follows by induction.
		
		To prove \eqref{eq:SumOfn-FoldTwinCoefficients-FixI-P0-SameOrientation}, again by \cite[Propositions 5--7]{H09}, we have
		\begin{align}
			\sum\limits_{\substack{p\in P,\; p<i_0 \\ j>i_0}} a_{\bfx_{p,j}}^{(n+1)}
			&= \biggl( \sum\limits_{\substack{p\in P,\; p<i_0 \\ j>i_0}} a_{\bfx_{p,j}}^{(n)} \biggr)
			\biggl( \sum\limits_{\substack{p\in P,\; p<i_0 \\ j>i_0}} a_{\bfx_{p,j}} + \sum\limits_{\substack{\bfd\in \Phi^+_{\mathrm{th}} \\ \bfd>\mathbf{1}_{D^{i_0}}}} a_\bfd \biggr)\nonumber\\&\qquad + \biggl( \sum\limits_{\substack{\bfd\in \Phi^+_{\mathrm{th}} \\ \bfd>\mathbf{1}_{D^{i_0}}}} a_\bfd^{(n)} \biggr)
			\biggl( \sum\limits_{\substack{p\in P,\; p<i_0 \\ j>i_0}} a_{\bfx_{p,j}} \biggr). \label{eq:SumOfn-FoldTwinCoefficients-FixI-P0-SameOrientation-RecursiveFormula}
		\end{align}
		
		By Proposition \ref{prop:SumOfCoefficientsOfGeneratorsOfThinRepOfTypeAandD}, we have
		\begin{equation}\label{eq:SumOfn-FoldTwinCoefficients-FixI-P0-SameOrientation-1}
			\sum\limits_{\substack{\bfd\in \Phi^+_{\mathrm{th}} \\ \bfd>\mathbf{1}_{D^{i_0}}}} a_\bfd^{(n)}
			= \sum_{\bfd\in S(\mathbf{1}_{D^{i_0+1}})} a_\bfd^{(n)}
			= \biggl( \sum\limits_{\substack{\bfd\in \Phi^+_{\mathrm{th}} \\ \bfd>\mathbf{1}_{D^{i_0}}}} a_\bfd \biggr)^{\!n}.
		\end{equation}
		
		Using \eqref{eq:SumOfn-FoldTwinCoefficients-FixI-P0-SameOrientation-1} and the inductive hypothesis in \eqref{eq:SumOfn-FoldTwinCoefficients-FixI-P0-SameOrientation-RecursiveFormula}, we obtain
		\begin{align*}
			\sum\limits_{\substack{p\in P,\; p<i_0 \\ j>i_0}} a_{\bfx_{p,j}}^{(n+1)}
			&= \Biggl( \biggl( \sum\limits_{\substack{p\in P,\; p<i_0 \\ j>i_0}} a_{\bfx_{p,j}} + \sum\limits_{\substack{\bfd\in \Phi^+_{\mathrm{th}} \\ \bfd>\mathbf{1}_{D^{i_0}}}} a_\bfd \biggr)^{\!n}
			- \biggl( \sum\limits_{\substack{\bfd\in \Phi^+_{\mathrm{th}} \\ \bfd>\mathbf{1}_{D^{i_0}}}} a_\bfd \biggr)^{\!n} \Biggr)
			\biggl( \sum\limits_{\substack{p\in P,\; p<i_0 \\ j>i_0}} a_{\bfx_{p,j}} + \sum\limits_{\substack{\bfd\in \Phi^+_{\mathrm{th}} \\ \bfd>\mathbf{1}_{D^{i_0}}}} a_\bfd \biggr) \\
			&\quad + \biggl( \sum\limits_{\substack{\bfd\in \Phi^+_{\mathrm{th}} \\ \bfd>\mathbf{1}_{D^{i_0}}}} a_\bfd \biggr)^{\!n}
			\biggl( \sum\limits_{\substack{p\in P,\; p<i_0 \\ j>i_0}} a_{\bfx_{p,j}} \biggr) \\
			&= \biggl( \sum\limits_{\substack{p\in P,\; p<i_0 \\ j>i_0}} a_{\bfx_{p,j}} + \sum\limits_{\substack{\bfd\in \Phi^+_{\mathrm{th}} \\ \bfd>\mathbf{1}_{D^{i_0}}}} a_\bfd \biggr)^{\!n+1}
			- \biggl( \sum\limits_{\substack{\bfd\in \Phi^+_{\mathrm{th}} \\ \bfd>\mathbf{1}_{D^{i_0}}}} a_\bfd \biggr)^{\!n+1},
		\end{align*}
		where
		\[
		\biggl( \sum\limits_{\substack{\bfd\in \Phi^+_{\mathrm{th}} \\ \bfd>\mathbf{1}_{D^{i_0}}}} a_\bfd \biggr)^{\!n+1}
		=\sum\limits_{\substack{\bfd\in \Phi^+_{\mathrm{th}} \\ \bfd>\mathbf{1}_{D^{i_0}}}} a_\bfd^{(n+1)}
		\]
		by  the same way in \eqref{eq:SumOfn-FoldTwinCoefficients-FixI-P0-SameOrientation-1}. 
		Hence \eqref{eq:SumOfn-FoldTwinCoefficients-FixI-P0-SameOrientation} follows by induction.
	\end{proof}
	
	\begin{lemma}\label{lem:SumOfn-FoldTwinCoefficients-FixI-DifferentOrientation}
		If $\sigma(\alpha)\neq \sigma(\beta)$, then for each (fixed) $1\leq i\leq l-3$, we have
		\begin{equation}\label{eq:SumOfn-FoldTwinCoefficients-FixI-DifferentOrientation}
			\sum_{j>i} a_{\bfx_{i,j}}^{(n)} = \biggl( \sum\limits_{\substack{i'\ge i \\ j'>i'}} a_{\bfx_{i',j'}} + \sum\limits_{\substack{\bfd\in \Phi^+_{\mathrm{th}} \\ \bfd>\mathbf{1}_{D^i}}} a_\bfd \biggr)^{\!n} - \biggl( \sum\limits_{\substack{i'>i \\ j'>i'}} a_{\bfx_{i',j'}} + \sum\limits_{\substack{\bfd\in \Phi^+_{\mathrm{th}} \\ \bfd>\mathbf{1}_{D^i}}} a_\bfd \biggr)^{\!n}.
		\end{equation}
	\end{lemma}
	
	\begin{proof}
		Equation \eqref{eq:SumOfn-FoldTwinCoefficients-FixI-DifferentOrientation} is proved by induction on $n$.
		The case $n=1$ is trivial.
		Assume that \eqref{eq:SumOfn-FoldTwinCoefficients-FixI-DifferentOrientation} holds for $n$.
		By \cite[Propositions 5--7]{H09}, we have
		\begin{equation}\label{eq:SumOfn-FoldTwinCoefficients-FixI-DifferentOrientation-Recursive}
			\sum_{j>i} a_{\bfx_{i,j}}^{(n+1)} = \Biggl( \sum_{j>i} a_{\bfx_{i,j}}^{(n)} \Biggr)
			\Biggl( \sum\limits_{\substack{i'\ge i \\ j'>i'}} a_{\bfx_{i',j'}} + \sum\limits_{\substack{\bfd\in \Phi^+_{\mathrm{th}} \\ \bfd>\mathbf{1}_{D^i}}} a_\bfd \Biggr)
			+ \Biggl( \sum\limits_{\substack{i'> i \\ j'>i'}} a_{\bfx_{i',j'}}^{(n)} + \sum\limits_{\substack{\bfd\in \Phi^+_{\mathrm{th}} \\ \bfd>\mathbf{1}_{D^i}}} a_\bfd^{(n)} \Biggr)
			\Biggl( \sum_{j>i} a_{\bfx_{i,j}} \Biggr).
		\end{equation}
		
		By \cite[Propositions 5--7]{H09} and Proposition \ref{prop:SumOfCoefficientsOfGeneratorsOfThinRepOfTypeAandD}, considering $\mathbf{1}_{D^{i+1}}$, we have
		\begin{equation}\label{eq:SumOfn-FoldTwinCoefficients-FixI-DifferentOrientation-2}
			\sum\limits_{\substack{i'> i \\ j'>i'}} a_{\bfx_{i',j'}}^{(n)} + \sum\limits_{\substack{\bfd\in \Phi^+_{\mathrm{th}} \\ \bfd>\mathbf{1}_{D^i}}} a_\bfd^{(n)} =\sum\limits_{\bfd\in S(\mathbf{1}_{D^{i+1}})}a_\bfd^{(n)}= \biggl( \sum\limits_{\substack{i'> i \\ j'>i'}} a_{\bfx_{i',j'}} + \sum\limits_{\substack{\bfd\in \Phi^+_{\mathrm{th}} \\ \bfd>\mathbf{1}_{D^i}}} a_\bfd \biggr)^{\!n}.
		\end{equation}
		
		Using the inductive assumption together with \eqref{eq:SumOfn-FoldTwinCoefficients-FixI-DifferentOrientation-2}, from \eqref{eq:SumOfn-FoldTwinCoefficients-FixI-DifferentOrientation-Recursive} we obtain
		\begin{align*}
			\sum_{j>i} a_{\bfx_{i,j}}^{(n+1)}
			={}& \Biggl( \biggl( \sum\limits_{\substack{i'\ge i \\ j'>i'}} a_{\bfx_{i',j'}} + \sum\limits_{\substack{\bfd\in \Phi^+_{\mathrm{th}} \\ \bfd>\mathbf{1}_{D^i}}} a_\bfd \biggr)^{\!n}
			- \biggl( \sum\limits_{\substack{i'>i \\ j'>i'}} a_{\bfx_{i',j'}} + \sum\limits_{\substack{\bfd\in \Phi^+_{\mathrm{th}} \\ \bfd>\mathbf{1}_{D^i}}} a_\bfd \biggr)^{\!n} \Biggr)
			\Biggl( \sum\limits_{\substack{i'\ge i \\ j'>i'}} a_{\bfx_{i',j'}} + \sum\limits_{\substack{\bfd\in \Phi^+_{\mathrm{th}} \\ \bfd>\mathbf{1}_{D^i}}} a_\bfd \Biggr) \\
			&+ \biggl( \sum\limits_{\substack{i'> i \\ j'>i'}} a_{\bfx_{i',j'}} + \sum\limits_{\substack{\bfd\in \Phi^+_{\mathrm{th}} \\ \bfd>\mathbf{1}_{D^i}}} a_\bfd \biggr)^{\!n}
			\Biggl( \sum_{j>i} a_{\bfx_{i,j}} \Biggr) \\
			={}& \biggl( \sum\limits_{\substack{i'\ge i \\ j'>i'}} a_{\bfx_{i',j'}} + \sum\limits_{\substack{\bfd\in \Phi^+_{\mathrm{th}} \\ \bfd>\mathbf{1}_{D^i}}} a_\bfd \biggr)^{\!n+1}
			- \biggl( \sum\limits_{\substack{i'>i \\ j'>i'}} a_{\bfx_{i',j'}} + \sum\limits_{\substack{\bfd\in \Phi^+_{\mathrm{th}} \\ \bfd>\mathbf{1}_{D^i}}} a_\bfd \biggr)^{\!n+1}.
		\end{align*}
		This completes the induction.
	\end{proof}
	
	\begin{proposition}\label{cor:SumOfn-FoldTwinCoefficients}
		The following statements hold.
		\begin{enumerate}
			\item If $\sigma(\alpha)=\sigma(\beta)$, let $P=\{1\leq k\leq l-3\mid \sigma(\gamma_{k})\neq \sigma(\alpha)\}$. Then
			\begin{align}\label{eq:SumOfTwinRoots-SameOrientation}
				\sum\limits_{\bfd\in \Phi^+_{\mathrm{tw}}}a_\bfd^{(n)}
				={}& n\sum_{\substack{i\in P}} \Biggl( \biggl( \sum_{j>i} a_{\bfx_{i,j}} \biggr)
				\biggl( \sum_{\substack{p\in P,\; p< i \\ j>i}} a_{\bfx_{p,j}} + \sum_{\substack{\bfd\in \Phi^+_{\mathrm{th}} \\ \bfd>\mathbf{1}_{D^i}}} a_\bfd \biggr)^{\!n-1} \Biggr) \nonumber \\
				&+ \sum_{\substack{i\notin P}} \Biggl( \biggl( \sum_{\substack{p\in P,\; p\le i \\ j>i}} a_{\bfx_{p,j}} + \sum_{\substack{\bfd\in \Phi^+_{\mathrm{th}} \\ \bfd>\mathbf{1}_{D^i}}} a_\bfd \biggr)^{\!n}
				- \biggl( \sum_{\substack{p\in P,\; p< i \\ j>i}} a_{\bfx_{p,j}} + \sum_{\substack{\bfd\in \Phi^+_{\mathrm{th}} \\ \bfd>\mathbf{1}_{D^i}}} a_\bfd \biggr)^{\!n} \Biggr).
			\end{align}
			
			\item If $\sigma(\alpha)\neq \sigma(\beta)$, then
			\begin{equation}\label{eq:SumOfTwinRoots-differentOrientation}
				\sum\limits_{\bfd\in \Phi^+_{\mathrm{tw}}}a_\bfd^{(n)}
				= \sum_{i=1}^{l-3} \Biggl(
				\biggl( \sum_{\substack{i'\ge i \\ j'>i'}} a_{\bfx_{i',j'}} + \sum_{\substack{\bfd\in \Phi^+_{\mathrm{th}} \\ \bfd>\mathbf{1}_{D^i}}} a_\bfd \biggr)^{\!n}
				- \biggl( \sum_{\substack{i'>i \\ j'>i'}} a_{\bfx_{i',j'}} + \sum_{\substack{\bfd\in \Phi^+_{\mathrm{th}} \\ \bfd>\mathbf{1}_{D^i}}} a_\bfd \biggr)^{\!n}
				\Biggr).
			\end{equation}
		\end{enumerate}
	\end{proposition}
	
	\begin{proof}
		The formulas \eqref{eq:SumOfTwinRoots-SameOrientation} and \eqref{eq:SumOfTwinRoots-differentOrientation} are deduced from Lemmas \ref{lem:SumOfn-FoldTwinCoefficients-FixI-SameOrientation}--\ref{lem:SumOfn-FoldTwinCoefficients-FixI-DifferentOrientation}.
	\end{proof}


\end{document}